\documentclass[11pt,letterpaper,twoside,reqno,nosumlimits,final]{amsart}

\usepackage[usenames,dvipsnames]{xcolor}
\usepackage{fancyhdr}
\usepackage{amsmath,amsfonts,amsbsy,amsgen,amscd,mathrsfs,amssymb,amsthm}
\usepackage{stackengine}
\usepackage{pifont}

\usepackage{mathtools}

\usepackage{graphicx}
\graphicspath{{figures/}}

\usepackage{tikz}
\usetikzlibrary{arrows,chains,matrix,positioning,scopes}
\usepackage{tikz-3dplot}

\usepackage{xcolor} %

\definecolor{cit}{rgb}{0.91,0.39,0.16}	%

\definecolor{dark-gray}{gray}{0.3}

\definecolor{dkgray}{rgb}{.3,.3,.3}
\definecolor{medgray}{rgb}{.5,.5,.5}
\definecolor{ltgray}{rgb}{.7,.7,.7}
\definecolor{dkblue}{rgb}{0,0,.5}
\definecolor{medblue}{rgb}{0,0,.75}
\definecolor{ltblue}{rgb}{0.97,0.97,1}
\definecolor{rust}{rgb}{0.5,0.1,0.1}
\definecolor{ltyellow}{rgb}{1, 1, 0.9}

\usepackage{soul}
\sethlcolor{ltyellow}
\newcommand{\hilite}[1]{\hl{#1}}

\usepackage[normalem]{ulem}

\usepackage[
style=alphabetic, citestyle=alphabetic,
giveninits=true,
sorting=nyt, sortcites=false,
autopunct=true, autolang=hyphen,
hyperref=true,
abbreviate=false,
backref=false,
backend=bibtex
]{biblatex}
\defbibheading{bibempty}{}

\renewbibmacro{in:}{%
  \ifentrytype{article}{}{\printtext{\bibstring{in}\intitlepunct}}}

\AtEveryBibitem{%
	\clearfield{issn}%
	\clearfield{isbn}%
	\clearlist{location}%
	\ifentrytype{book}%
		{%
		\clearfield{series}%
		\clearfield{volume}%
		\clearfield{pages}}%
		{%
		\clearname{editor}}%
}

\renewbibmacro*{doi+eprint+url}{%
  \newunit\newblock
  \iftoggle{bbx:doi}
    {\printfield{doi}}
    {}%
  \iftoggle{bbx:eprint}
    {\iffieldundef{doi}{\usebibmacro{eprint}}{}}
    {}%
  \newunit\newblock
  \iftoggle{bbx:url}
    {\iffieldundef{doi}{\iffieldundef{eprint}{\usebibmacro{url}}}} %
    {}}

\AtBeginBibliography{\footnotesize}				%

\usepackage{url} %
\makeatletter
\g@addto@macro{\UrlBreaks}{\UrlOrds}
\makeatother

\usepackage{hyperref}

\hypersetup{hidelinks,colorlinks=true,breaklinks=true,urlcolor=cit,citecolor=dkblue,linkcolor=dkblue,bookmarksopen=false,hypertexnames=false}

\usepackage{bookmark}
\bookmarksetup{
open,
numbered,
addtohook={%
\ifnum\bookmarkget{level}=0 %
\bookmarksetup{bold}%
\fi
\ifnum\bookmarkget{level}=-1 %
\bookmarksetup{color=cit,bold}%
\fi
}
}

\usepackage[full]{textcomp}

\usepackage[scaled=.98,sups,osf]{XCharter}%
\usepackage[scaled=1.04,varqu,varl]{inconsolata}%
\usepackage[type1,scaled=0.98]{cabin}%
\usepackage[utopia,vvarbb,scaled=1.07]{newtxmath}
\usepackage[cal=euler]{mathalfa}
\usepackage{bm} %

\usepackage{microtype} %
\usepackage[utf8]{inputenc} %
\usepackage[T1]{fontenc} %

\usepackage[english]{babel} %
\usepackage{setspace}

\usepackage{enumitem}

\setlist{noitemsep} %

\setlist[enumerate]{font=\sffamily\bfseries\footnotesize\textcolor{dkgray},label=\arabic*.}
\setlist[itemize]{font=\textcolor{dkgray},label=\small\textcolor{dkgray}\textbullet}

\usepackage{thmtools}
\usepackage[capitalize]{cleveref}

\numberwithin{equation}{section}

\theoremstyle{theorem}

\newtheorem{bigthm}{Theorem}

\newtheorem{theorem}{Theorem}[section]
\newtheorem{lemma}[theorem]{Lemma}

\newtheorem{proposition}[theorem]{Proposition}
\newtheorem{fact}[theorem]{Fact}

\newtheorem{corollary}[theorem]{Corollary}

\theoremstyle{definition}

\newtheorem{definition}[theorem]{Definition}
\newtheorem{example}[theorem]{Example}
\newtheorem{remark}[theorem]{Remark}

\newtheorem*{ideaT}{Theme}
\newtheorem*{problemT}{Problem}

\Crefname{proposition}{Proposition}{Propositions}

\RequirePackage[framemethod=default]{mdframed}

\newmdenv[skipabove=6pt,
skipbelow=6pt,
rightline=false,
leftline=true,
topline=false,
bottomline=false,
backgroundcolor=ltyellow,
linecolor=cit,
innerleftmargin=10pt,
innerrightmargin=10pt,
innertopmargin=0pt,
innerbottommargin=5pt,
leftmargin=0cm,
rightmargin=0cm,
linewidth=4pt]{iBox}

\newmdenv[skipabove=0pt,
skipbelow=0pt,
backgroundcolor=ltblue,
linecolor=dkblue,
linewidth=2pt,
rightline=false,
leftline=false,
topline=false,
bottomline=false,
innerleftmargin=7pt,
innerrightmargin=10pt,
innertopmargin=6pt,
innerbottommargin=6pt,
leftmargin=0cm,
rightmargin=0cm,
innerbottommargin=5pt]{aBox}

\newmdenv[skipabove=10pt,
skipbelow=10pt,
backgroundcolor=white,
linecolor=dkblue,
linewidth=0.5pt,
rightline=true,
leftline=true,
topline=true,
bottomline=true,
innerleftmargin=10pt,
innerrightmargin=0.5in,
innertopmargin=5pt,
innerbottommargin=5pt,
leftmargin=0cm,
rightmargin=0cm]{lfBox}

\usepackage{float}
\floatstyle{plaintop}

\numberwithin{figure}{section}
\numberwithin{table}{section}

\newfloat{recipe}{thp}{lor}
\floatname{recipe}{Recipe}
\numberwithin{recipe}{section}

\usepackage{caption}
\usepackage{subcaption}

\usepackage{booktabs,longtable,tabu} %
\usepackage{multirow} %

\usepackage{algorithm,algorithmicx}
\usepackage[noend]{style/myalgpseudocode}

\algrenewcommand\alglinenumber[1]{\sf\scriptsize\color{dkgray}{#1}}
\algrenewcommand\algorithmicrequire{\textbf{Input:}}
\algrenewcommand\algorithmicensure{\textbf{Output:}}

\algdef{SE}[SUBALG]{Indent}{EndIndent}{}{\algorithmicend\ }
\algtext*{Indent}
\algtext*{EndIndent}

\newcommand{\econst}{\mathrm{e}}
\newcommand{\iunit}{\mathrm{i}}

\renewcommand{\phi}{\varphi}

\newcommand{\onecirc}{\text{\ding{192}}}
\newcommand{\twocirc}{\text{\ding{193}}}

\newcommand{\vct}[1]{\bm{#1}}
\newcommand{\mtx}[1]{\bm{#1}}
\newcommand{\set}[1]{\mathsf{#1}}

\newcommand{\coll}[1]{\mathcal{#1}}

\newcommand{\term}[1]{\textit{\hilite{#1}}}

\newcommand{\N}{\mathbb{N}}
\newcommand{\Z}{\mathbb{Z}}

\newcommand{\R}{\mathbb{R}}

\newcommand{\C}{\mathbb{C}}

\newcommand{\M}{\mathbb{M}}
\newcommand{\Sym}{\mathbb{H}}

\newcommand{\sgn}{\operatorname{sgn}}
\renewcommand{\Re}{\operatorname{Re}}
\renewcommand{\Im}{\operatorname{Im}}

\newcommand{\ntr}{\operatorname{tr}}
\newcommand{\diag}{\operatorname{diag}}
\newcommand{\Id}{\mathbf{I}}

\newcommand{\spec}{\operatorname{spec}}
\newcommand{\northeast}{\operatorname{northeast}}

\newcommand{\abs}[1]{\vert {#1} \vert}
\newcommand{\norm}[1]{\Vert {#1} \Vert}
\newcommand{\ip}[2]{\langle {#1}, \ {#2} \rangle}

\newcommand{\lip}[2]{\left\langle {#1}, \ {#2} \right\rangle}
\newcommand{\labs}[1]{\left\vert {#1} \right\vert}
\newcommand{\lnorm}[1]{\left\Vert {#1} \right\Vert}

\newcommand{\dist}{\mathrm{dist}}

\newcommand{\diff}{\mathrm{d}}
\newcommand{\Diff}{\mathrm{D}}
\newcommand{\idiff}{\,\diff}
\newcommand{\ddx}[1]{\frac{\diff}{\diff{#1}}}

\newcommand{\Expect}{\operatorname{\mathbb{E}}}
\newcommand{\Var}{\operatorname{Var}}
\newcommand{\Cov}{\operatorname{Cov}}

\newcommand{\Probe}{\mathbb{P}}

\newcommand{\condbar}{\, \vert \,}
\newcommand{\lcondbar}{\, \big\vert \,}

\newcommand{\Varo}{\operatorname{Var_{\otimes}}}

\newcommand{\normal}{\textsc{normal}}
\newcommand{\uniform}{\textsc{uniform}}

\DeclareFontFamily{U}{matha}{\hyphenchar\font45}
\DeclareFontShape{U}{matha}{m}{n}{
  <-6> matha5 <6-7> matha6 <7-8> matha7
  <8-9> matha8 <9-10> matha9
  <10-12> matha10 <12-> matha12
  }{}

\DeclareSymbolFont{matha}{U}{matha}{m}{n}
\DeclareMathSymbol{\abscont}{3}{matha}{"CE}

\makeatletter
\def\paragraph{\@startsection{paragraph}{4}%
  \z@\z@{-\fontdimen2\font}%
  {\normalfont\scshape}}
\makeatother

\addglobalbib{Tro26-Universality.bib} %

\usepackage[foot]{amsaddr}

\title[Universality for Random Matrices]{Universality Laws for Random Matrices \\ via Exchangeable Counterparts}

\author{Joel A.~Tropp}
\address[JAT]{Department of Computing and Mathematical Science, Caltech, Pasadena, CA, USA.}
\email{jtropp@caltech.edu, https://tropp.caltech.edu}

\date{Research from August 2022.  Revised: 17 July 2024 and 29 February 2026.}


\begin{document}

\begin{abstract}
Recently, Brailovskaya \& van Handel (\textsl{GAFA}, 2024)
established a suite of nonasymptotic universality laws
which demonstrate that the spectral statistics
of an independent sum of random matrices
mirror the spectral statistics
of a Gaussian random matrix
with the same first- and second-order moments.
This paper develops a more elementary proof
of their main results
by means of a new implementation
of the method of exchangeable counterparts.
\end{abstract}

\maketitle

\section{Main results and related work}

Random matrix theory has become a central part of
contemporary mathematics---pure, applied, and computational.
Therefore, it has become increasingly urgent to develop systematic
techniques for studying the complicated random matrices
that arise in practice.

In a recent influential paper,
Brailovskaya \& van Handel~\cite{BvH24:Universality-Sharp}
proved that the spectral properties of
an independent sum of random matrices align with
the spectral properties of a Gaussian random matrix
that shares the same first- and second-order moments.
Their arguments rely on a difficult
implementation of Stein's method that requires infinite
cumulant expansions, M{\"o}bius inversion,
high-order derivatives of matrix functions,
and multivariate trace inequalities.

The purpose of this paper is to develop a more elementary
approach to these universality theorems.  Our arguments
are also based on techniques from Stein's method.
We have designed a new variant of the method of
exchangeable counterparts that allows us to avoid
the high-order expansions that are inherent in
the existing analysis.  Our techniques may make it
more transparent why the universality theorems
hold, and they may support future refinements and
extensions.

\subsection{Independent sums}

Consider an independent family $(\mtx{S}_1, \dots, \mtx{S}_n)$
of random \hilite{self-adjoint} matrices with common dimension $d$,
either real or complex.  Form the independent sum
\begin{equation} \label{eqn:indep-sum-intro}
\mtx{X} \coloneqq \sum_{i=1}^n \mtx{S}_i \in \Sym_d.
\end{equation}
The symbol $\Sym_d$ denotes the real-linear space of $d \times d$
self-adjoint matrices.
Our goal is to develop tools for understanding the spectral properties
of the independent sum $\mtx{X}$.
First, we wish to establish probability inequalities
for the spectral support that tell us where to find
eigenvalues of $\mtx{X}$.
Second, we want to approximate trace functions of the matrix $\mtx{X}$,
which provide information about the distribution of the eigenvalues.

For simplicity, we only treat random self-adjoint matrices.
We can establish similar results for the singular values
of random rectangular matrices using the self-adjoint
dilation~\cite[Sec.~2.1.17]{Tro15:Introduction-Matrix}.
We omit these standard arguments.

\subsection{The Gaussian proxy}

This paper establishes \term{universality theorems}---statements
that the spectral properties of the sum $\mtx{X}$ depend
primarily on the first- and second-order moments of the
summands but not strongly on their detailed distributions.

To that end, introduce a \hilite{Gaussian} self-adjoint matrix $\mtx{Z} \in \Sym_d$
that shares the same first- and second-order
moments with the independent sum $\mtx{X}$.  That is,
\begin{equation} \label{eqn:gauss-intro}
\mtx{Z} \sim \normal( \Expect[\mtx{X}], \Varo[\mtx{X}] ).
\end{equation}
We will argue that the spectral statistics of the Gaussian matrix
$\mtx{Z}$ mirror the spectral statistics of the independent sum $\mtx{X}$,
provided that each summand $\mtx{S}_i$ is relatively small.
These results are powerful because we gain access to a rich trove of
methods for studying Gaussian random matrix models.

Recall that a random self-adjoint matrix $\mtx{Z} \in \Sym_d$ %
is \term{Gaussian} when the real and imaginary parts
of its entries $\{ (\Re Z_{jk}, \Im Z_{jk}) : j, k = 1,\dots,d\}$
compose a jointly Gaussian family of real random variables.
A Gaussian self-adjoint matrix is fully characterized
by its entrywise expectation $\Expect[ \mtx{Z} ]$
and its variance tensor
\begin{equation} \label{eqn:Varo}
\Varo[ \mtx{Z} ] \coloneqq
	\Expect[ (\mtx{Z} - \Expect \mtx{Z}) \otimes (\mtx{Z} - \Expect \mtx{Z}) ],
\end{equation}
where $\otimes$ denotes the Kronecker product.
For example, see~\cite[Sec.~2.1.2]{BvH24:Universality-Sharp}
or~\cite[Sec.~2]{Tro26:Comparison-Theorems-max}.

\subsection{Random matrix statistics}
\label{sec:statistics-intro}

We frame the main results in terms of two simple statistics of
the independent sum model~\eqref{eqn:indep-sum-intro}.
In most cases, we can compute both statistics with short linear
algebra arguments.
First, the \term{matrix variance} is the statistic
\[
\sigma^2(\mtx{X}) \coloneqq \lnorm{ \Expect\big[ (\mtx{X} - \Expect \mtx{X})^2 \big] }
	= \lnorm{ \sum_{i=1}^n \Expect\big[ (\mtx{S}_i - \Expect \mtx{S}_i)^2 \big] }.
\]
The symbol $\norm{\cdot}$ denotes the \term{spectral norm},
also known as the $\ell_2$ operator norm.
The matrix variance describes the scale on which the independent sum
fluctuates around its expectation \hilite{as a matrix}.
Second, the \term{uniform bound} statistic for the summands is
\[
L(\mtx{X}) \coloneqq \sup \max\nolimits_i \norm{\mtx{S}_i - \Expect \mtx{S}_i}.
\]
This quantity reflects the maximum deviation of any summand from its
expected value.  The notation $L(\mtx{X})$ is slightly misleading
because the uniform bound statistic depends strongly on the decomposition
of $\mtx{X}$ into independent summands.

The matrix Bernstein inequality~\cite[Thm.~6.1.1]{Tro15:Introduction-Matrix}
states that these two statistics provide rough control on the spectral
norm of an independent sum $\mtx{X} \in \Sym_d$ of
random self-adjoint matrices: %
\begin{equation} \label{eqn:mtx-bernstein}
\tfrac{1}{2} \sqrt{\sigma^2(\mtx{X})}
	\quad\leq\quad \Expect \norm{\mtx{X} - \Expect \mtx{X}}
	\quad\leq\quad \sqrt{2 \sigma^2(\mtx{X}) \log(2d)}\ +\ \tfrac{1}{3} L(\mtx{X}) \log(2d).
\end{equation}
Up to constants, both inequalities are attained by specific types
of random matrices, but it can be hard to ascertain whether the
lower bound or the upper bound is closer to the truth.

For each positive number $q > 0$,
define the $\set{L}_q$ (quasi)norm of a random matrix $\mtx{Y}$ via the rule
\begin{equation} \label{eqn:Lq-norm}
\norm{\mtx{Y}}_{q} \coloneqq \left( \Expect \ntr \abs{\mtx{Y}}^q \right)^{1/q}.
\end{equation}
Here and elsewhere, $\ntr$ denotes the \term{normalized trace},
and $\abs{\mtx{A}} \coloneqq (\mtx{A}^*\mtx{A})^{1/2}$ returns the positive
part of the polar decomposition of a matrix $\mtx{A}$.
The function $q \mapsto \norm{\mtx{Y}}_q$ is increasing.

\subsection{Monomial moments}

The first main result states that the even-order moments of the
independent sum~\eqref{eqn:indep-sum-intro} are consistent with
the even-order moments of the Gaussian proxy~\eqref{eqn:gauss-intro}.

\begin{bigthm}[Monomial moments: Universality] \label{thm:mom-intro} %
Let $\mtx{X}$ be an independent sum~\eqref{eqn:indep-sum-intro}
of random self-adjoint matrices, and let $\mtx{Z}$ be the matching
Gaussian model~\eqref{eqn:gauss-intro}.
For each natural number $p \in \N$,
\[
\labs{ \norm{\mtx{X}}_{2p} - \norm{\mtx{Z}}_{2p} }
	\leq \big( 16 p^2 \sigma^2(\mtx{X}) L(\mtx{X}) \big)^{1/3} + 16 p L(\mtx{X}).
\]
Furthermore, when $p L^2(\mtx{X}) \leq \sigma^2(\mtx{X})$, %
\[
\labs{ \norm{\mtx{X}}_{2p} - \norm{\mtx{Z}}_{2p} }
	\leq 32 p^2 L(\mtx{X}). %
\]
The statistics $\sigma^2(\mtx{X})$ and $L(\mtx{X})$ are defined
in \cref{sec:statistics-intro}.
\end{bigthm}

\noindent
\Cref{thm:mom-intro} is a variant of \cite[Thm.~2.9]{BvH24:Universality-Sharp}.
The proof appears in \cref{sec:mono-mom}, where we establish a result with more
refined statistics.

\Cref{thm:mom-intro} is informative when $L(\mtx{X}) \ll \sigma(\mtx{X})$;
this condition ensures that no summand makes an outsize contribution to the sum.
With additional work, the result implies a probabilistic comparison
between the spectral norm $\norm{\mtx{X}}$ of the independent sum
and the spectral norm $\norm{\mtx{Z}}$ of the Gaussian proxy.
It also yields some information about the agreement between their
eigenvalue distributions.
We will not elaborate on these points because more detailed statements
follow from the other main results.

\subsection{Mean spectral distribution}

For a random self-adjoint matrix $\mtx{Y} \in \Sym_d$, %
the \term{mean spectral distribution (msd)} is
the unique probability measure
$\overline{\mu}_{\mtx{Y}}$ on the real line
that is determined by the equations
\begin{equation} \label{eqn:msd}
\Expect \ntr h(\mtx{Y})
	\coloneqq \frac{1}{d} \sum_{i=1}^d \Expect h(\lambda_i(\mtx{Y}))
	\eqqcolon \int_{\R} h(y) \, \overline{\mu}_{\mtx{Y}}(\diff{y})
	\quad\text{for all $h : \R \to \R_+$.}
\end{equation}
The eigenvalues $\lambda_i$ are arranged in decreasing order,
and we require that the positive function $h$ is Borel measurable.
Instead of working directly with the msd,
let us define the \term{Cauchy transform}:
\begin{equation} \label{eqn:cauchy-xform}
G_\zeta(\mtx{Y}) \coloneqq \Expect \ntr\big[ (\zeta \Id_d - \mtx{Y})^{-1} \big]
\quad\text{for each $\zeta \in \C \setminus \R$.}
\end{equation}
The Cauchy transform fully determines the msd
of the random matrix~\cite[Sec.~2.4.3]{Tao12:Topics-Random}

The second main result states that the Cauchy transform of the
independent sum~\eqref{eqn:indep-sum-intro} is consistent with the
Cauchy transform of the Gaussian proxy~\eqref{eqn:gauss-intro}.

\begin{bigthm}[Cauchy transform: Universality] \label{thm:cauchy-xform}
Let $\mtx{X}$ be an independent sum~\eqref{eqn:indep-sum-intro}
of random self-adjoint matrices, and let $\mtx{Z}$ be the matching
Gaussian model~\eqref{eqn:gauss-intro}.
Then their Cauchy transforms satisfy
\[
\abs{ G_\zeta( \mtx{X} ) - G_\zeta(\mtx{Z}) }
	\leq \frac{4 \sigma^2(\mtx{X}) L(\mtx{X})}{\abs{\Im \zeta}^4.}
	\quad\text{for each $\zeta \in \C \setminus \R$.}
\]
The statistics $\sigma^2(\mtx{X})$ and $L(\mtx{X})$ are defined
in \cref{sec:statistics-intro}.
\end{bigthm}

\noindent
\Cref{thm:cauchy-xform} is a variant of
\cite[Thm.~2.10 and Rem.~6.13]{BvH24:Universality-Sharp}.
The proof appears in \cref{sec:cauchy-xform}.

\Cref{thm:cauchy-xform} implies a universality result for
smooth spectral functions.
This claim follows from an argument of
Haagerup \& Thorbj{\o}rnsen~\cite[Sec.~6]{HT05:New-Application},
summarized in~\cite[Sec.~5.2.3]{BvH24:Universality-Sharp}.

\begin{corollary}[Spectral functions: Universality]
Instate the assumptions of \cref{thm:cauchy-xform}.
For each sufficiently smooth function $h : \R \to \R$,
\[
\labs{ \Expect \ntr h(\mtx{X}) - \Expect\ntr h(\mtx{Z}) }
	\lesssim \sigma^2(\mtx{X}) L(\mtx{X}) \cdot \norm{h}_{\set{W}_{1,5}}.
\]
The order $\lesssim$ suppresses universal constants,
and the $\set{W}_{1,5}$ Sobolev norm returns the total of the $\set{L}_1$ norms
of the first $5$ derivatives.
\end{corollary}

\subsection{Spectral support}

The msd does not control individual eigenvalues, but only their
overall distribution.
To obtain more precise information about the location of the
spectrum, we work with resolvent matrices.
For a self-adjoint matrix $\mtx{A} \in \Sym_d$, %
the \term{resolvents} are the normal matrices
\[
R_\zeta(\mtx{A}) \coloneqq (\zeta \Id_d - \mtx{A})^{-1}
\quad\text{for each $\zeta \in \C \setminus \R$.}
\]
The size of the resolvent $R_{\zeta}(\mtx{A})$ reflects (the reciprocal of)
the distance from $\zeta$ to the eigenvalues of $\mtx{A}$.

The last main result states that $\set{L}_p$ norms
of the resolvents of the independent sum~\eqref{eqn:indep-sum-intro} are
consistent with the $\set{L}_p$ norms of the resolvents of the Gaussian
proxy~\eqref{eqn:gauss-intro}.

\begin{bigthm}[Resolvent norm: Universality] \label{thm:resolvent-norm}
Let $\mtx{X}$ be an independent sum~\eqref{eqn:indep-sum-intro}
of random self-adjoint matrices, %
and let $\mtx{Z}$ be the matching Gaussian model~\eqref{eqn:gauss-intro}.
For each $\zeta \in \C \setminus \R$ and for each natural number $p \in \N$,
\[
\labs{ \norm{ R_{\zeta}(\mtx{X}) }_{2p} - \norm{ R_{\zeta}( \mtx{Z} ) }_{2p} }
	\leq \frac{ 48 p^2 \sigma^2(\mtx{X}) L(\mtx{X}) + 72 p^3 L(\mtx{X})^3}{\abs{\Im \zeta}^4}.
\]
The statistics $\sigma^2(\mtx{X})$ and $L(\mtx{X})$ are defined
in \cref{sec:statistics-intro}.
\end{bigthm}

\noindent
\Cref{thm:resolvent-norm} is a variant of \cite[Thm.~6.8]{BvH24:Universality-Sharp}.
The proof appears in \cref{sec:resolvent-norm}.

\Cref{thm:resolvent-norm} implies that the spectral support of
the independent sum is consistent with the spectral support of
the Gaussian proxy with respect to the Hausdorff distance,
$\dist_{\rm H}$.  This statement follows from a lengthy
concentration argument appearing in~\cite[Sec.~7]{BvH24:Universality-Sharp}

\begin{corollary}[Spectrum: Universality]
Instate the assumptions of \cref{thm:resolvent-norm}.
Then
\[
\Expect \dist_{\rm H}(\spec(\mtx{X}), \spec(\mtx{Z}))
	\lesssim \big( \sigma^2(\mtx{X}) L(\mtx{X}) \log d \big)^{1/3}
	+ L(\mtx{X}) \log d + \big(\sigma_*^2(\mtx{X}) \log d \big)^{1/2}. 
\]
The \term{weak variance} is defined as
\[
\sigma_*^2(\mtx{X}) \coloneqq \sup\nolimits_{\norm{\vct{u}}=1} \Var[ \vct{u}^* \mtx{X} \vct{u} ]
	= \sup\nolimits_{\norm{\vct{u}}=1} \sum_{i=1}^n \Var[ \vct{u}^* \mtx{S}_i \vct{u} ].
\]
\end{corollary}

\subsection{Universality: Proof strategies}

In the paper of Brailovskaya \& van Handel~\cite{BvH24:Universality-Sharp}
and in the present work,
the universality statements are derived from versions of Stein's method.
This section summarizes the similarities and differences between these
two approaches.
To simplify some formulas, we assume that each of the
summands is centered: $\Expect[\mtx{S}_i]=\mtx{0}$.

\subsubsection{Interpolation}

The first step follows a standard procedure in high-dimensional probability.
To compare the independent sum~\eqref{eqn:indep-sum-intro}
and the matching Gaussian model~\eqref{eqn:gauss-intro},
we interpolate: %
\[
\mtx{Y}_t \coloneqq %
	\sqrt{t} \, \mtx{X} %
	+ \sqrt{1-t} \, \mtx{Z} %
	\quad\text{for $t \in [0,1]$.}
\]
The expectation and the variance tensor remain constant along
the interpolation path. %

Each of the main results in this paper compares the expected
value of a trace function at each of the two random matrix models.
To perform this comparison for a smooth function $h : \R \to \C$, %
we introduce the interpolants
\[
u(t) \coloneqq \Expect \ntr h (\mtx{Y}_t)
\quad\text{for $t = [0,1]$.}
\]
Setting $f \coloneqq \Diff h$, the derivative $\dot{u}(t)$ along
the interpolation path takes the form
\begin{equation} \label{eqn:udot-intro}
\dot{u}(t) = \frac{1}{2\sqrt{t}} \Expect \ntr[ \mtx{X} f(\mtx{Y}_t) ]
	- \frac{1}{2\sqrt{1-t}} \Expect \ntr[ \mtx{Z} f(\mtx{Y}_t) ].
\end{equation}
The challenge is to compare the two covariance-like expressions that
appear in~\eqref{eqn:udot-intro}.
The goal of the analysis is to reach
a differential inequality of the form
\begin{equation} \label{eqn:diff-ineq-intro}
\abs{ \dot{u}(t) } \leq C \cdot \abs{ u(t) }^{1 - \theta}
\quad\text{for all $t \in (0,1)$.} %
\end{equation}
Given~\eqref{eqn:diff-ineq-intro}, we quickly obtain a comparison
between $u(1) = \Expect\ntr h(\mtx{X})$ and $u(0) = \Expect\ntr h(\mtx{Z})$
by solving the differential inequality.

\subsubsection{Cumulant expansions}

To study the derivative $\dot{u}(t)$ along the interpolation path,
Brailovskaya \& van Handel~\cite{BvH24:Universality-Sharp}
employ the method of cumulant expansions, attributed to
Barbour~\cite{Bar86:Asymptotic-Expansions} and Lytova \& Pastur~\cite{LP09:Central-Limit}.
The key result~\cite[Thm.~4.3]{BvH24:Universality-Sharp} is a beautiful formula:
\[
\dot{u}(t) = \frac{1}{2} \sum_{p=3}^\infty \sum_{i=1}^n \sum_{\substack{j_1,\dots,j_p \\ k_1,\dots,k_p}}
	\frac{t^{p/2-1}}{(p-1)!} \cdot \kappa_p( S_{ij_1k_1}, \cdots, S_{ij_pk_p})
	\cdot \Expect \left[ \frac{\partial^p \ntr h (\mtx{Y}_t)}{\partial S_{ij_1k_1} \cdots \partial S_{ij_pk_p}} \right].
\]
In this expression, $\kappa_p$ denotes the $p$th-order multilinear cumulant,
and $S_{ijk}$ refers to the $(j, k)$ entry of the summand $\mtx{S}_i$.
The formula holds when $h : \R \to \C$ is a polynomial function.

Let us pause for a moment to elaborate. %
Since the independent sum $\mtx{X}$ and the Gaussian proxy $\mtx{Z}$ have matching
first- and second-order moments, the first two cumulants of entries $\mtx{X}$
equal the first two cumulants of entries of $\mtx{Z}$.
Thus, the derivative along the interpolation path only depends on
the cumulants of order three and higher.
Since the cumulant vanishes when two of its arguments are statistically independent,
we only need to treat cumulants involving the entries of individual summands.
The remarkable feature of this approach is
that it separates the statistical properties of the summands
from the smoothness properties of the trace function.

To take advantage of this description of the derivative $\dot{u}(t)$,
Brailovskaya \& van Handel require several devices.
First, to treat functions that are not polynomials, such as resolvents,
they truncate the series and bound the error terms (their Theorem 4.4).
Second, to revert from cumulants to moments, they
employ M{\"o}bius inversion (their Corollary 6.1).
Third, to evaluate the expectation, %
they must compute high-order derivatives of the matrix function
(e.g., their Lemma 6.3).
Fourth, to reach the differential inequality~\eqref{eqn:diff-ineq-intro}, they
use a sophisticated trace inequality (their Proposition 5.1).

There is a sense in which their approach is mechanical,
which makes it potentially applicable to a range of problems.
On the other hand, this set of arguments is both technical and involved.
This complexity make it challenging to appreciate the reason
that the universality results hold true and to extend the strategy
to other situations.

\subsubsection{Exchangeable counterparts}

This paper develops a more elementary approach that charts
a route around some of the difficulties inherent in the
cumulant expansion.  the argument is based on the method
of exchangeable counterparts, which originates in the work of
Charles Stein~\cite{Ste86:Approximate-Computation} on
normal approximation.  The application to random matrices
was introduced in the paper~\cite{MJCFT14:Matrix-Concentration}.

Our strategy uses exchangeable counterparts of the
independent sum~\eqref{eqn:indep-sum-intro} to capture
its fluctuations.  For example, suppose that we construct
an exchangeable counterpart $\mtx{X}'$ to the sum $\mtx{X}$
by replacing a randomly chosen summand with an independent copy.
For any function $\mtx{F} : \Sym_d \to \M_d$,
we have the covariance identity
\[
\Expect[ \mtx{X} \cdot \mtx{F}(\mtx{X}) ]
	= \frac{n}{2} \Expect[ (\mtx{X} - \mtx{X}') (\mtx{F}(\mtx{X}) - \mtx{F}(\mtx{X}')) ].
\]
Using this representation, we can obtain a closed form
representation of the derivative~\eqref{eqn:udot-intro}
using only \hilite{third-order} differences of the function $\mtx{F}$.
This process applies to any function.
It allows us to avoid taking high-order derivatives
and invoking elaborate trace inequalities.
On the other hand, it is more delicate to work with differences
(as compared with derivatives), and the argument may seem
less automatic than the cumulant expansion.
We also require an extra technical input, namely the matrix
Rosenthal inequality (\cref{fact:rosenthal-mtx}),
although that result is fairly easy to
prove~\cite[Thm.~7.1]{MJCFT14:Matrix-Concentration}.

Our strategy has some features in common with Stein's method
for multivariate normal approximation;
for example, see~\cite{Mec09:Steins-Method,Rol13:Steins-Method}.
Nevertheless, the details of our approach are novel, even
in the scalar setting.
Our implementation allows us to reach the differential
inequality~\eqref{eqn:diff-ineq-intro}, even for
nonconvex functions, such as resolvent powers.

\subsection{Other related work}

The technique of comparing a random matrix with a Gaussian proxy
is a standard move in random matrix theory that long predates
the paper~\cite{BvH24:Universality-Sharp}.  For example,
see the tutorial paper~\cite{Tao19:Least-Singular}.

Very recently, the author has shown that it is possible to
\hilite{compare} the extreme eigenvalues of an
independent sum of random matrices with the extreme eigenvalues
of a Gaussian proxy~\cite{Tro25:Comparison-Theorems,Tro26:Comparison-Theorems-max}.
These techniques only lead to one-sided bounds,
rather than a two-sided equivalence in the spirit of \cref{thm:mom-intro}.
The benefit is that the hypotheses of the comparison
theorems are weaker, and they can produce stronger guarantees
than the universality results presented here.

\subsection{Roadmap}

\Cref{sec:cov-scalar,sec:univ-scalar} develop our
universality arguments in the scalar setting,
where the strategy may be more transparent.
\Cref{sec:mtx-calculus} introduces elements
of the difference calculus for matrix functions,
and \cref{sec:mtx-solid} presents some new trace
inequalities that play a role in the proof.
\Cref{sec:cov-matrix,sec:interp-matrix} generalize
the scalar universality tools to the matrix setting.
Afterward, we develop proofs of the three main
results in \cref{sec:cauchy-xform,sec:mono-mom,sec:resolvent-norm}.
\Cref{app:scalar-diff} summarizes facts about scalar divided
differences, and \cref{app:rosenthal} documents
Rosenthal inequalities for scalars and matrices.

\subsection{Notation}

The results for random matrices are valid in both
the real field $\R$ or the complex field $\C$.
The star ${}^*$ returns the complex conjugate of a complex number
or the (conjugate) transpose of a matrix.
The symbols $\vee$ and $\wedge$ denote the
infix maximum and minimum.
Nonlinear functions bind before the trace and
the expectation, and we may omit brackets when
they are unnecessary.

The upright $\Diff$ always refers to the derivative
operator, while $\Delta$ refers to the divided difference
operator.  For a function $u(t)$ of a real variable $t$, we
also use the dot notation $\dot{u}(t)$ for the derivative.
Primes always refer to exchangeable counterparts---never derivatives.

The linear space $\M_d$ contains all $d \times d$ matrices
over the scalar field.
For a matrix $\mtx{A} \in \M_d$, we write
$\ntr[\mtx{A}] \coloneqq d^{-1} \sum_{i=1}^d a_{ii}$
for the \term{normalized trace}.
The unadorned $\norm{\cdot}$ always refers
to the \term{spectral norm}.
Throughout the paper, we implicitly use the
normalized Schatten norm $\set{S}_q$ for $q \geq 1$, defined
by $\norm{\mtx{A}}_q \coloneqq (\ntr \abs{\mtx{A}}^q)^{1/q}$. %
As noted, $\abs{\mtx{A}}$ returns the positive part of
the polar decomposition. %
The spectrum, $\spec(\mtx{A})$, is the set of all complex
eigenvalues of the matrix.

Recall that a square matrix $\mtx{A}$ is
\term{self-adjoint} when $\mtx{A} = \mtx{A}^*$.
The real-linear subspace $\Sym_d$ collects the
self-adjoint matrices in $\M_d$.
For a function $f : \R \to \R$ and a self-adjoint matrix $\mtx{A}$,
the \term{standard matrix function} $f(\mtx{A})$
is computed by applying $f$ to each eigenvalue
of $\mtx{A}$ without changing the associated eigenspace.

The operator $\Probe(\cdot)$ returns the probability of an
event, and $\Expect[\cdot]$ computes the expectation of a
random variable that takes values in a linear space.
The functions $\Var$ and $\Cov$ compute the variance and
covariance of real random variables.
For $q \geq 1$, the symbol $\norm{\cdot}_q$ refers to the $\set{L}_q$ norm
of a real random variable or the $\set{L}_p(\set{S}_p)$
norm of a matrix-valued random variable,
as in~\eqref{eqn:Lq-norm}.

We occasionally use the order symbol $\lesssim$, which
suppresses a universal constant.

\section{Covariance identities: Scalar setting}
\label{sec:cov-scalar}

We begin with a universality theory
for independent sums of \hilite{real} random variables.
In this setting, the arguments are simpler,
so the ideas shine through more brightly.
The universality results for random matrices
follow the same pattern of argument, but we
must address several additional complications.

Our approach is based on covariance
identities that can be regarded as (discrete)
integration by parts (IBP) rules.
We establish these results using the method of exchangeable counterparts,
a core tool in Stein's method~\cite{Ste86:Approximate-Computation}.
Nevertheless, the specific techniques in this paper
appear to be novel. %
For more background on Stein's method, see the
monographs~\cite{CGS11:Normal-Approximation,Ros11:Fundamentals-Steins}.

All of our results hold under weak regularity conditions.
We will explicitly frame assumptions about smoothness of functions.
On the other hand, we tacitly assume that all expectations
are defined and finite, without spelling out the requirement each time.

\subsection{Gaussian IBP}

The Gaussian IBP rule allows us to compute the covariance between a
Gaussian random variable and a function of the same Gaussian random variable.
This result provides a template for covariance identities for other
types of random variables.

\begin{fact}[Gaussian IBP] \label{fact:gauss-ibp}
Let $Z$ be a \hilite{real} Gaussian random variable.
For each continuously differentiable function $f : \R \to \R$,
\[
\Cov(Z, f(Z)) = \Var[Z] \cdot \Expect[ \Diff f(Z) ].
\]
The symbol $\Diff$ denotes the derivative operator.
\end{fact}

\begin{proof}[Proof sketch]
Express the covariance as an integral with respect
to the Gaussian distribution, and apply
integration by parts.  See~\cite[Lem.~1.1.1]{NP12:Normal-Approximations}
for details.
\end{proof}

\subsection{Exchangeability}

We plan to develop an analog of the Gaussian IBP rule that can shed
light on the distribution of an independent sum.  We begin with a
fundamental definition.

\begin{definition}[Exchangeable family]
A family $(X_1, \dots, X_k)$ of random variables is \term{exchangeable}
if the distribution $(X_{\pi(1)}, \dots, X_{\pi(k)})$ is the same
for every permutation $\pi$.  Equivalently, for any integrable function $f$,
\[
\Expect[ f(X_1, \dots, X_k) ] = \Expect[ f(X_{\pi(1)}, \dots, X_{\pi(k)}) ]
\quad\text{for each permutation $\pi$.}
\]
In particular, each of the random variables $X_i$ has the same marginal distribution.  
We use the prime symbol $(X, X', X'')$ to denote exchangeable counterparts
of a random variable.  It is sometimes helpful to abbreviate a function
with permuted coordinates as $f^{\pi} \coloneqq f \circ \pi$.
\end{definition}

For every random variable $X$, we can construct two fundamental exchangeable pairs:
\begin{itemize}
\item	\textbf{Trivial pair.}  The pair $(X,X)$ is exchangeable.
\item	\textbf{Independent pair.} If $X'$ is an \hilite{independent} copy of $X$,
then $(X, X')$ is exchangeable.
\end{itemize}
These are the extreme examples.  It is more useful to construct exchangeable
counterparts that are induced by a small modification to the random variable.

\subsection{Key example: Independent sums}
\label{sec:indep-sum-exch}

Consider an \hilite{independent} family $(S_1, \dots, S_n)$ of
real random variables.  Introduce the sum
\begin{equation} \label{eqn:indep-sum}
X \coloneqq \sum_{i=1}^n S_i.
\end{equation}
To construct an exchangeable counterpart for $X$, draw an independent copy
$(S_1', \dots, S_n')$ of the sequence $(S_1, \dots, S_n)$ of summands.  Next,
draw a random variable $I \sim \uniform\{1, \dots, n\}$, independent from
everything else.  Remove the $I$th summand from the sum and replace it
by a fresh copy:
\begin{equation} \label{eqn:indep-sum-exch}
X' \coloneqq X + (S_I' - S_I) = S_I' + \sum_{j \neq I} S_j.
\end{equation}
It is not hard to see that the pair $(X, X')$ is exchangeable,
although the two constituents are typically not independent. %

For jointly Gaussian random variables, conditional expectation amounts
to linear regression.   A related property holds for the exchangeable pair %
we just constructed.

\begin{proposition}[Independent sum: Linear regression property] \label{prop:indep-sum-linreg}
As in~\eqref{eqn:indep-sum}, consider an independent sum $X$ with $n$ terms,
and construct the exchangeable counterpart $X'$ as in~\eqref{eqn:indep-sum-exch}. %
Then
\begin{equation} \label{eqn:indep-sum-linreg}
\Expect[ X - X' \condbar X ] = n^{-1} (X - \Expect X).
\end{equation}
\end{proposition}

\begin{proof}
From the definition of $(X, X')$, we determine that
\begin{align*}
\Expect[ X - X' \condbar S_1, \dots, S_n ]
	&= \Expect[ S_I - S_I' \condbar S_1, \dots, S_n ]
	= \Expect\left[ \frac{1}{n} \sum_{i=1}^n (S_i - S_i') \lcondbar S_1, \dots, S_n \right] \\
	&= \frac{1}{n} \sum_{i=1}^n (S_i - \Expect S_i) = n^{-1} (X - \Expect X).
\end{align*}
Take the conditional expectation with respect to $X$ to complete the argument.
\end{proof}

We can construct more exchangeable counterparts for~\eqref{eqn:indep-sum} by a similar process.
Let $(S_1'', \dots, S_n'')$ be another independent copy of the summands.
Define
\[
X'' \coloneqq X + (S_I'' - S_I) = S_I'' + \sum_{j \neq I} S_j.
\]
To be clear, we use the \hilite{same} random index $I$ to generate both $X'$ and $X''$.
The resulting triple $(X, X', X'')$ is exchangeable,
even if we condition on the value of $I$.

\subsection{Discrete covariance identities}

The linear regression property (\cref{prop:indep-sum-linreg}) supports
an elegant covariance identity for independent sums.

\begin{proposition}[Covariance identity: Independent sum] \label{prop:discrete-ibp-sum}
As in \cref{sec:indep-sum-exch}, let $X$ be a real independent sum~\eqref{eqn:indep-sum}
with $n$ terms, and let $X'$ be the exchangeable counterpart~\eqref{eqn:indep-sum-exch}.
For each function $f : \R \to \R$,
\begin{equation} \label{eqn:discrete-ibp-sum}
\Cov( X, f(X) ) = \frac{n}{2} \Expect[ (X - X')( f(X) - f(X') ) ].
\end{equation}
\end{proposition}

\begin{proof}
Using the linear regression property~\eqref{eqn:indep-sum-linreg},
\begin{align*}
\Cov( X, f(X) ) &= \Expect[ (X - \Expect X) f(X) ] \\
	&= n \cdot \Expect[ \Expect[X-X' \condbar X] \cdot f(X) ]
	= n \cdot \Expect[ (X-X') f(X) ]. 
\end{align*}
The pair $(X, X')$ is exchangeable, so we can swap their roles:
\[
\Cov( X, f(X) ) = \Cov(X', f(X')) = n \cdot \Expect[ (X'-X) f(X') ].
\]
Average the two displays to complete the proof.
\end{proof}

In particular, \cref{prop:discrete-ibp-sum} contains a fundamental
variance identity.  Taking $f(t) \coloneqq t$, %
\begin{equation} \label{eqn:var-exch}
\Var[X] = \Cov(X, X) = \Expect[ (X - \Expect X)^2 ] = \frac{n}{2} \Expect[ (X - X')^2 ].
\end{equation}
We can determine the variance of the sum $X$ by comparing
it with its exchangeable counterpart $X'$.

\subsection{Divided differences}

Let us rewrite the discrete covariance identity (\cref{prop:discrete-ibp-sum}) as
\begin{equation} \label{eqn:indep-sum-cov-dd}
\Cov(X, f(X)) = \Expect\left[ \frac{n}{2} (X - X')^2 \cdot \frac{f(X) - f(X')}{X - X'} \right]
\end{equation}
This expression admits a strong analogy with the Gaussian IBP rule (\cref{fact:gauss-ibp}).
In view of~\eqref{eqn:var-exch}, the first term in the expectation approximates the variance of $X$, while the second term approximates the derivative $f'(X)$.
To pursue this observation, we introduce \term{divided differences}.

\begin{definition}[Divided differences]
Let $f : \R \to \R$ be a continuously differentiable function.  The \term{divided difference operator}
$\Delta$ acts on $f$ to produce a \hilite{continuous} bivariate function:
\[
\Delta f(a_0, a_1) \coloneqq \frac{f(a_0) - f(a_1)}{a_0 - a_1}
\quad\text{for distinct $a_0, a_1 \in \R$.}
\]
By continuity, $\Delta f(a, a) = \Diff f(a)$ for $a \in \R$.

Higher-order divided differences are defined by iteration.  Assume that
$f : \R \to \R$ is $k$-times continuously differentiable for a natural number $k \geq 2$.
For real points $a_0, \dots, a_k \in \R$,
\[
\Delta^k f(a_0, \dots, a_k) \coloneqq
	\frac{\Delta^{k-1} f(a_0, \dots, a_{k-1}) - \Delta^{k-1} f(a_1, \dots, a_k)}{a_0 - a_k}
	\quad\text{for distinct $a_0, a_k \in \R$.}
\]
The confluent case $a_0 = a_k$ is determined by the requirement that $\Delta^k f$
is continuous.
Divided differences and derivatives are linked by the formula
$\Delta^k f(a, \dots, a) = \Diff^k f(a) / k!$ for all $a \in \R$.
\end{definition}

Here is a simple example that has significance for us.

\begin{example}[Power function: Divided differences]
Fix a natural number $p \in \N$, and consider the power function $f(a) \coloneqq a^p$
for $a \in \R$.  For all $a, b \in \R$, we can confirm the algebraic identity
\[
f(a) - f(b) = a^p - b^p = \sum_{q=0}^{p-1} a^q (a-b) b^{p-1-q}.
\]
Indeed, the sum telescopes.  It follows that the first divided difference satisfies
\[
\Delta f(a, b) = \sum_{q + r = p-1} a^q b^r
\quad\text{where $q, r \in \Z_+$.}
\]
By a parallel argument, for all $a, b, c \in \R$, the second divided difference takes the form
\[
\Delta^2 f(a, b, c) = \sum_{q+r+s = p-2} a^q b^r c^s
\quad\text{where $q,r,s \in \Z_+$.}
\]
Higher-order divided differences follow a similar pattern.
\end{example}

The universality arguments require some additional facts about scalar divided differences,
collected in \cref{app:scalar-diff}.  We give detailed cross-references
as needed.  See~\cite{deB05:Divided-Differences} or~\cite[App.~B.16]{Hig08:Functions-Matrices}
for further information.

\subsection{Discrete IBP}

With this preparation, we can develop a discrete IBP rule for the
independent sum that parallels the Gaussian IBP rule.

\begin{theorem}[Discrete IBP: Independent sum] \label{thm:discrete-ibp-sum}
As in \cref{sec:indep-sum-exch}, consider the exchangeable triple $(X,X',X'')$
induced by a real independent sum with $n$ terms.
For each function $f : \R \to \R$ with two continuous derivatives,
\begin{align*}
\Cov(X, f(X)) &= \Var[X] \cdot \Expect[ \Diff f(X) ] \\
	&+ \frac{1}{2} \Expect\big[ W(X,X',X'') \cdot \big( \Delta^2 f (X,X',X'') + \Delta^2 f(X,X,X') \big) \big]
\end{align*}
where the random variable
$%
W(X,X',X'') \coloneqq n \cdot (X' - X'')^2 (X' - X).
$%
\end{theorem}

Compare this result with the Gaussian IBP rule (\cref{fact:gauss-ibp}).
For an independent sum, we suffer an error term that reflects the smoothness of the
function through $\Delta^2 f$ and the first three moments of the summands through $W$.

\begin{proof}
We begin with the formulation~\eqref{eqn:indep-sum-cov-dd} of the covariance identity
in terms of the divided difference:
\[
\Cov(X, f(X)) = \frac{n}{2} \Expect[ (X - X')^2 \cdot \Delta f(X, X') ].
\]
To the divided difference, add and subtract $\Delta f(X'', X'') = \Diff f(X'')$
to reach
\[
\Cov(X, f(X)) = \frac{n}{2} \Expect[ (X - X')^2 \cdot \Diff f(X'') ]
	+ \frac{n}{2} \Expect\big[ (X-X')^2 \big( \Delta f(X,X') - \Delta f(X'',X'') \big) \big]
\]
Let us revise this formula to make it more actionable.

For the first term, note that $X''$ is independent from $X - X' = S_I - S_I'$.
Thus,
\[
\frac{n}{2} \Expect[ (X - X')^2 \cdot \Diff f(X'') ]
	= \frac{n}{2} \Expect[ (X - X')^2 ] \cdot \Expect[ \Diff f(X'') ]
	= \Var[X] \cdot \Expect[ \Diff f(X) ].
\]
We have employed the variance identity~\eqref{eqn:var-exch},
and we have made the exchange $X'' \leftrightarrow X$.

For the second term, we reduce the differences to second order.  To do so,
add and subtract $\Delta f(X, X'')$.  This step yields
\[
\Delta f(X,X') - \Delta f(X'',X'')
	= (X'- X'') \cdot \Delta^2 f(X,X',X'') + (X - X'') \cdot \Delta^2 f(X,X'',X'').
\]
Substitute this identity into the expectation to obtain
\begin{align*}
\Expect[ (X-X')^2 \cdot \big( \Delta f(X,X') - \Delta f(X'',X'') \big)
	&= \Expect[ (X-X')^2 (X'-X'') \cdot \Delta^2 f(X,X',X'') ] \\
	&+ \Expect[ (X-X')^2 (X-X'') \cdot \Delta^2 f(X,X'',X'') ].
\end{align*}
To clean up, we can make the exchange $(X, X', X'') \rightarrow (X'', X', X)$
in the first summand and $(X, X', X'') \rightarrow (X', X'', X)$
in the second summand, and then recall that the divided difference
is a symmetric function of its arguments.
Reassemble the pieces the reach the conclusion.
\end{proof}

\section{Universality: Scalar setting}
\label{sec:univ-scalar}

In this section, we develop a framework for establishing
universality results for an independent sum
of real random variables.
As in~\cite{Mec09:Steins-Method},
the strategy is to interpolate between the sum and a matching
Gaussian model.  The new insight is to express the error terms
with divided differences.
After presenting the basic results, we explain how they lead
to a quantitative central limit theorem and a comparison for
monomial moments.

\subsection{Interpolation}

Consider a real random variable $X$ with at least two finite moments.
Introduce an \hilite{independent} Gaussian random variable
$Z \sim \normal_{\R}(\Expect[X], \Var[X])$
with the same expectation and variance.
To compare a function of $X$ with a function of $Z$,
we carefully interpolate between the random
variables and control the rate of change along the path of interpolation.

Define a family of interpolants:
\begin{equation} \label{eqn:interp}
Y_t \coloneqq Y_t(X, Z) \coloneqq \Expect[X] + \sqrt{t} \cdot (X - \Expect X) + \sqrt{1-t} \cdot (Z- \Expect Z)
\quad\text{for all $t \in [0,1]$.}
\end{equation}
Observe that $Y_1 = X$ and $Y_0 = Z$.  This formulation ensures that the expectation
and the variance both remain constant on the interpolation path:
\[
\Expect[ Y_t ] = \Expect[X]
\quad\text{and}\quad
\Var[Y_t ] = \Var[X]
\quad\text{for all $t \in [0,1]$.}
\]
We also note the form of the derivatives.  For $t \in (0,1)$,
\[
\ddx{t} Y_t(x, z) = \frac{x-\Expect X}{2\sqrt{t}} - \frac{z - \Expect Z}{2\sqrt{1-t}}
\quad\text{and}\quad
\ddx{x} Y_t(x, z) = \sqrt{t}
\quad\text{and}\quad
\ddx{z} Y_t(x, z) = \sqrt{1-t}.
\]
Our plan is to bound derivatives of functions along the path.
We begin with an easy result.

\begin{proposition}[Interpolation] \label{prop:interp}
As above, introduce real random variables $X$ and $Z$ along with the interpolants
$Y_t$ from~\eqref{eqn:interp}.  Consider a continuously differentiable
function $h : \R \to \R$ with derivative $f \coloneqq \Diff h$, and define
\[
u(t) \coloneqq \Expect h(Y_t)
\quad\text{for all $t \in [0,1]$.}
\]
For $t \in (0,1)$, the time derivative $\dot{u}(t)$ takes the form
\begin{equation} \label{eqn:interp-cov}
\dot{u}(t) = \frac{ \Cov(X, f(Y_t(X,Z)))}{2\sqrt{t}} - \frac{ \Cov(Z, f(Y_t(X,Z)))}{2\sqrt{1-t}}.
\end{equation}
\end{proposition}

\begin{proof} %
Use dominated convergence to pass the derivative through the expectation.
Apply the chain rule, and identify the covariances.
\end{proof}

\subsection{Interpolation: Independent sum}

The derivative~\eqref{eqn:interp-cov} along the interpolation path exposes two covariances.
When the random variable $X$ is an independent sum, we can employ the
discrete IBP rule (\cref{thm:discrete-ibp-sum}) to treat its covariance.

\begin{proposition}[Interpolation: Independent sum] \label{prop:interp-sum}
As in \cref{sec:indep-sum-exch}, consider the exchangeable triple $(X,X',X'')$
induced by a real independent sum with $n$ terms.  Draw an independent
Gaussian random variable $Z \sim \normal_{\R}(\Expect[X], \Var[X])$. %

Let $h : \R \to \R$ be three times continuously differentiable, with derivative $f \coloneqq \Diff h$.
Define the interpolant $u(t) \coloneqq \Expect h(Y_t)$ where $Y_t$ is
given by~\eqref{eqn:interp}.
For $t \in (0,1)$, the derivative $\dot{u}(t)$ along the interpolation path satisfies
\begin{equation} \label{eqn:interp-diff}
\dot{u}(t) = \frac{\sqrt{t}}{4} \Expect\big[ W(X, X',X'') \cdot \big( \Delta^2 f(Y_t, Y_t', Y_t'') + \Delta^2 f(Y_t, Y_t, Y_t') \big) \big]
\end{equation}
where the exchangeable counterparts $Y_t' \coloneqq Y_t(X', Z)$ and $Y_t'' \coloneqq Y_t(X'', Z)$
and
\[
W(X,X',X'') \coloneqq n \cdot (X' - X'')^2 (X' - X).
\]
\end{proposition}

\begin{proof}
\Cref{prop:interp} provides a formula~\eqref{eqn:interp-cov} for the derivative $\dot{u}$
along the interpolation path.  We use integration by parts to rewrite the two covariance terms.
The Gaussian IBP rule (\cref{fact:gauss-ibp}) for the function $z \mapsto f(Y_t(X,z))$ yields
\begin{align*}
\frac{\Cov(Z, f(Y_t(X,Z))) }{2\sqrt{1-t}} = \frac{1}{2} \Var[X] \cdot \Expect[ \Diff f(Y_t(X,Z)) ].
\end{align*}
We have employed the chain rule and the fact that $\Var[Z] = \Var[X]$.
Meanwhile, the discrete IBP rule (\cref{thm:discrete-ibp-sum}) for the function
$x \mapsto f(Y_t(x,Z))$ provides that
\begin{multline*}
\frac{ \Cov(X, f(Y_t(X,Z)))}{2\sqrt{t}} = \frac{1}{2} \Var[X] \cdot \Expect[ \Diff f(Y_t(X,Z)) ] \\
+ \frac{\sqrt{t}}{4} \Expect\big[ W(X,X',X'') \cdot \big( \Delta^2 f(Y_t, Y_t', Y_t'') + \Delta^2 f(Y_t,Y_t,Y_t') \big) \big].
\end{multline*}
We have applied the chain rule (\cref{fact:diff-calculus}) for divided differences.  Introduce the displays into
the expression~\eqref{eqn:interp-cov} for $\dot{u}$ to complete the proof.
\end{proof}

\subsection{Application: Quantitative CLT}

As a warmup, we establish a quantitative central limit theorem (CLT)
for an independent sum of scalar random variables.
This statement is quite standard; we can obtain similar results
from a variety of arguments.

\begin{theorem}[Independent sum: Quantitative CLT] \label{thm:quant-clt}
Let $X$ be a real independent sum~\eqref{eqn:indep-sum} with $n$ terms,
and construct the matching Gaussian random variable $Z \sim \normal_{\R}(\Expect[X], \Var[X])$.
Define the third moment statistic
\begin{equation} \label{eqn:M3-scalar}
M_3 \coloneqq \sum_{i=1}^n \Expect \abs{ S_i - S_i' }^3,
\end{equation}
where $(S_1', \dots S_n')$ is an independent copy of the summands.
For each function $h : \R \to \R$ with three bounded continuous derivatives,
\[
\abs{ \Expect h(X) - \Expect h(Z) }
	\leq \frac{1}{6} \norm{ \Diff^3 h }_{\sup} \cdot {M}_3.
\]
\end{theorem}

After standard smoothing arguments~\cite[Sec.~7.2]{Pol02:Users-Guide}, %
we can obtain estimates in terms of the bounded Lipschitz distance,
which metrizes weak convergence:
\[
\dist_{\mathrm{BL}}(X, Z) \coloneqq \sup\big\{ \labs{\Expect h(X) - \Expect h(Z)} : \norm{h}_{\sup} \vee \norm{h}_{\mathrm{Lip}} \leq 1 \big\}
	\leq 4 {M}_3^{1/3}.
\]
In other words, the total of the third moments of the summands controls
the distance between the independent sum $X$ and a normal random
variable $Z$ with matching expectation and variance.

\subsubsection{Moment bounds}

We begin with a simple lemma that controls the expectation
of the random variable $W(X, X', X'')$.

\begin{lemma}[Discrete IBP: Moment bound] \label{lem:W-bd}
Under the assumptions of \cref{thm:quant-clt}, the random variable $W(X, X', X'')$
from \cref{prop:interp-sum} satisfies the bound
\[
\Expect \abs{W(X,X',X'')} \leq {M}_3.
\]
\end{lemma}

\begin{proof}
Apply Young's inequality and invoke exchangeability to obtain
\[
\Expect \abs{W(X,X',X'')}
	\leq n \cdot \left[ \frac{2}{3} \Expect \abs{X'-X''}^3 + \frac{1}{3} \Expect \abs{ X' - X }^3 \right]
	= n \cdot \Expect \abs{ X - X' }^3.
\]
Note that $X - X' = S_I - S_I'$, and take the expectation with respect to the random index $I$.
\end{proof}

\subsubsection{Proof of \cref{thm:quant-clt}}

Consider a function $h : \R \to \R$ whose first three derivatives are bounded and continuous,
and define the derivative $f \coloneqq \Diff h$.  Without loss, we may assume that
the Gaussian random variable $Z$ is independent from the independent sum $X$.

Construct the interpolant $u(t) \coloneqq \Expect h(Y_t)$ where $Y_t$ is defined in~\eqref{eqn:interp}.
\Cref{prop:interp-sum} provides a formula~\eqref{eqn:interp-diff} for the derivative $\dot{u}$
along the interpolation path.  Making a uniform bound on both of the second divided differences,
\[
\abs{ \dot{u}(t) } \leq \frac{\sqrt{t}}{2} \norm{ \Delta^2 f }_{\sup} \cdot \Expect{ \abs{W(X,X',X'')} }
	\leq \frac{\sqrt{t}}{4} \norm{ \Diff^3 h }_{\sup} \cdot {M}_3.
\]
To pass from the divided difference to the derivative, we employed~\eqref{eqn:diff-deriv-bd} and the relation $f = \Diff h$.
The bound for the expectation is \cref{lem:W-bd}.  By the fundamental theorem of calculus,
\[
\abs{ \Expect h(X) - \Expect h(Z) } = \abs{ u(1) - u(0) } \leq \int_0^1 \abs{ \dot{u}(t) } \idiff{t}
	\leq \frac{1}{6} \norm{ \Diff^3 h }_{\sup} \cdot {M}_3.
\]
This is the required result.
\hfill\qed

\subsection{Application: Monomial moments}

As a more sophisticated example of this machinery,
we compare the monomial moments of an independent sum
with the monomial moments of a matching Gaussian random variable.
This argument is more delicate because the monomials are unbounded,
and it takes full advantage of exchangeability.
In contrast to \cref{thm:quant-clt}, this result provides
a more precise relationship between the tails of the independent sum
and the normal model.

\begin{theorem}[Independent sum: Monomial moments] \label{thm:scalar-poly}
Let $X$ be a real independent sum~\eqref{eqn:indep-sum} with $n$ terms,
and construct the matching Gaussian random variable $Z \sim \normal_{\R}(\Expect[X], \Var[X])$.
Define the third moment statistic
\[
{M}_{3,p/3} \coloneqq \lnorm{ \sum_{i=1}^n \abs{S_i - S_i'}^3 }_{p/3},
\]
where $(S_1', \dots, S_n')$ is an independent copy of the summands.
For $p=2$ and for each real $p \geq 4$, the $p$th moments satisfy the comparisons
\begin{align} \label{eqn:scalar-poly-cube}
\labs{ \norm{ X }_{p} - \norm{ Z }_{p} }^3
	\leq \big\vert \norm{ X }_{p}^3 - \norm{ Z }_{p}^3 \big\vert
	\leq \frac{1}{2} (p-1)(p-2) \cdot {M}_{3,p/3}.
\end{align}
In particular, the relative error satisfies
\begin{align} \label{eqn:scalar-poly-rel}
\frac{\labs{ \norm{ X }_{p} - \norm{ Z }_{p} }}{\norm{Z}_p}
	\leq \frac{1}{2} (p-1)(p-2) \cdot \frac{{M}_{3,p/3}}{\norm{Z}_p^3}.
\end{align}
\end{theorem}

The third moment statistic $M_{3,p/3}$ may be hard to calculate directly.
Using Rosenthal inequalities (\cref{fact:rosenthal-scalar}),
we can easily obtain bounds that reflect the statistics of individual summands.

\begin{corollary}[Independent sum: Monomial moments] \label{cor:scalar-poly}
Instate the assumptions of \cref{thm:scalar-poly}.  Define statistics
\[
M_3 \coloneqq \sum_{i=1}^n \Expect \abs{ S_i - S_i' }^3
\quad\text{and}\quad
L_p \coloneqq \norm{ \max\nolimits_i \abs{S_i - S_i'} }_p.
\]
For each $p \geq 4$,
\[
\labs{ \norm{ X }_{p} - \norm{ Z }_{p} }
	\leq (p^2 M_3 )^{1/3} + p L_p.
\]
\end{corollary}

\begin{proof}[Proof sketch]
For $p \geq 4$, Rosenthal's inequality (\cref{fact:rosenthal-scalar}, positive case)
implies that
\begin{equation} \label{eqn:M3-bd-scalar}
M_{3,p/3} \leq \left[ M_3^{1/2} + \big(2 \vee \sqrt{p/3-1} \big)\, L_{p}^{3/2} \right]^2
	\leq 2 M_3 + 2p L_p^3.
\end{equation}
Substitute this bound into the moment comparison~\eqref{eqn:scalar-poly-cube}.
\end{proof}

\subsubsection{Example: Uniformly bounded summands}

As a concrete example, consider an independent sum $X$ of $n$ \hilite{standardized} real
random variables (i.e., each with expectation zero and variance one).  Moreover, assume that each summand is subject to the uniform bound $\abs{S_i} \leq L$.  We determine that
the statistics ${M}_3 \leq 6Ln$ and ${L}_{p} \leq 2L$,
while the Gaussian model satisfies $\norm{ Z }_{p} \geq \sqrt{pn / \econst}$.
The relative error estimate~\eqref{eqn:scalar-poly-rel}
and the bound~\eqref{eqn:M3-bd-scalar} for $M_{3,p/3}$ yield
\[
\frac{\labs{ \norm{X}_{p} - \norm{Z}_{p} }}{\norm{Z}_{p}}
	\leq \frac{p^2 M_3 + pL_p^3}{\norm{Z}_p^3}
	\leq \frac{6p^2 L n + 2p^3 L^3}{(pn/\econst)^{3/2}}
	\lesssim \left(\frac{p L^2}{n}\right)^{1/2} + \left(\frac{pL^2}{n}\right)^{3/2}.
\]
In case $pL^2 \ll n$, the comparison is nontrivial,
and we confirm the convergence rate $\mathcal{O}(n^{-1/2})$.

\subsubsection{Convexity bounds}

The analysis exploits convexity properties of the derivatives
of the power function to bound the derivative on the interpolation path.

\begin{lemma}[Discrete IBP: Convexity bound] \label{lem:indep-sum-cvx-bd}
Consider a function $f : \R \to \R$ for which $\abs{ \Diff^2 f }$ is convex.
Let $(X,X',X'')$ be an exchangeable triple.
For a function $Y \coloneqq Y(X)$ with exchangeable counterparts
$Y' \coloneqq Y(X')$ and $Y'' \coloneqq Y(X'')$ and a function
$W \coloneqq W(X,X',X'')$, we have the bound
\[
\Expect \labs{ W(X,X',X'') \cdot \Delta^2 f(Y,Y',Y'') }
	\leq \frac{1}{2} \max\nolimits_{\pi} \Expect\big[ \Expect\big[ \abs{ W^{\pi}(X,X',X'') } \lcondbar X \big] \cdot \abs{\Diff^2 f(Y)}\big]
\]
The identical upper bound holds with $\Delta^2 f(Y,Y,Y')$ in place of $\Delta^2 f(Y,Y',Y'')$.
\end{lemma}

Exchangeability is critical to obtain a bound that only depends
on the random variable $Y$ and not the counterparts $Y', Y''$.
The form of the result allows us to execute a moment comparison argument.

\begin{proof}
Using the Genocchi--Hermite formula (\cref{fact:diff-integral}) for the second divided difference, %
\begin{align*}
E &\coloneqq \Expect \abs{ W(X,X',X'') \cdot \Delta^2 f(Y,Y',Y'') } \\
	&\leq \int_{\set{T_2}} \Expect\big[ \abs{ W(X, X',X'') } \cdot \abs{ \Diff^2 f ( \tau_0 Y + \tau_1 Y' + \tau_2 Y'' ) } \big]\idiff{\vct{\tau}}.
\end{align*}
Since the function $\abs{\Diff^2 f}$ is convex, the integrand is a convex function of the triple
$\vct{\tau} = (\tau_0, \tau_1, \tau_2)$ on the simplex $\set{T}_2$.
Therefore, it attains its maximum when $\vct{\tau}$ is one of the standard
basis vectors.  The volume of $\set{T}_2$ is $1/2$, so we reach the bound
\[
E \leq \frac{1}{2} \max\nolimits_{i\in \{ 0,1, 2\}} \Expect\big[ \abs{ W(X,X',X'') }\cdot \abs{ \Diff^2 f( Y^{(i)} ) } \big].
\]
The notation $Y^{(i)}$ selects the $i$th exchangable counterpart, so it counts the number of prime symbols.
By exchangeability, we can introduce a permutation $\pi$ to swap $X^{(i)}$ with $X^{(0)} = X$,
resulting in an inequality of the form
\[
E \leq \frac{1}{2} \max\nolimits_{\pi} \Expect\big[ \abs{ W^{\pi}(X,X',X'')} \cdot \abs{ \Diff^2 f(Y) } \big].
\]
Last, we condition on $X$ to average out the extra randomness in $(X',X'')$.
\end{proof}

\subsubsection{Proof of \cref{thm:scalar-poly}}

The independent sum $X$ and the Gaussian proxy $Z$ share the same expectation and variance, so the statement of \cref{thm:scalar-poly} is trivially correct for $p = 2$.

Assume that $p \geq 4$.
Consider the power function $h(a) \coloneqq \abs{a}^{p}$
and the interpolants $u(t) \coloneqq \Expect h(Y_t) = \Expect \abs{Y_t}^{p}$.
The derivative of the power function $f(a) \coloneqq (\Diff h)(a) = p (\sgn a) \abs{a}^{p-1}$.
The second derivative of $f$ is continuous:
\begin{equation} \label{eqn:scalar-pow-der2}
(\Diff^2 f)(a) = p(p-1)(p-2) (\sgn a) \abs{a}^{p-3},
\end{equation}
The absolute value $\abs{\Diff^2 f}$ is evidently convex since $p \geq 4$. %

The formula~\eqref{eqn:interp-diff} for the derivative $\dot{u}(t)$ of the interpolant from \cref{prop:interp-sum} yields
\begin{align*}
\abs{ \dot{u}(t) } &\leq \frac{\sqrt{t}}{4} \Expect \labs{ W(X, X',X'') \cdot \big( \Delta^2 f(Y_t, Y_t', Y_t'') + \Delta^2 f(Y_t, Y_t, Y_t') \big) } \\
	&\leq \frac{\sqrt{t}}{4} \max\nolimits_{\pi} \Expect\big[ \Expect\big[ \abs{ W^{\pi}(X,X',X'') } \lcondbar X \big] \cdot \abs{\Diff^2 f(Y_t)} \big] \\
	&= \frac{p(p-1)(p-2)\sqrt{t}}{4} \max\nolimits_{\pi} \Expect\big[ \Expect\big[ \abs{ W^{\pi}(X,X',X'') } \lcondbar X \big] \cdot \abs{Y_t}^{p-3} \big].
\end{align*}
We have applied \cref{lem:indep-sum-cvx-bd} to pass from the second divided difference to the second derivative, which has the explicit expression~\eqref{eqn:scalar-pow-der2}.

To continue, invoke H{\"o}lder's inequality with exponents $\theta = 3/p$ and $\theta' = (p - 3)/p$.
This step brings forward a copy of the interpolant $u(t)$: %
\begin{equation} \label{eqn:scalar-pow-diff-ineq}
\begin{aligned}
\abs{ \dot{u}(t) } &\leq \frac{p(p-1)(p-2)\sqrt{t}}{2} \max\nolimits_{\pi} \norm{ \Expect[ \abs{ W^{\pi}(X,X',X'') } \condbar X ] }_{p/3} 
\cdot (\Expect \abs{Y_t}^{p} )^{1 - 3/p} \\
	&\leq \frac{p(p-1)(p-2)\sqrt{t}}{4} \cdot {M}_{3,p/3} \cdot \abs{u(t)}^{1 - 3/p}.
\end{aligned}
\end{equation}
To reach the second line, we have bounded the $\set{L}_{p/3}$ norm by means of a short argument:
\begin{multline*}
\norm{ \Expect[ \abs{ W^{\pi}(X,X',X'') } \condbar X ] }_{p/3}
	\leq \norm{ \Expect_I \abs{ W^{\pi}(X,X',X'')}  }_{p/3} \\
	= \norm{ \Expect_I \abs{ W(X,X',X'') } }_{p/3}
	= \lnorm{ \sum_{i=1}^n \abs{ (S_i'' - S_i')^2 (S_i'' - S_i) } }_{p/3}
	\leq {M}_{3,p/3}. 
\end{multline*}
The operator $\Expect_I$ computes the expectation with respect to the random
index $I$, and we have used exchangeability to remove the permutation.
The last step, based on Young's inequality, parallels the proof of~\cref{lem:W-bd}.

Finally, we solve the differential inequality~\eqref{eqn:scalar-pow-diff-ineq}.
Introduce the auxiliary function $\phi(t) \coloneqq u(t)^{3/p}$ for $t \in [0,1]$.
By the fundamental theorem of calculus and the chain rule,
\begin{align*}
\abs{ u(1)^{3/p} - u(0)^{3/p} }
	&= \abs{ \phi(1) - \phi(0) }
	\leq \int_0^1 \labs{ \ddx{t} \phi(t) } \idiff{t}
	= \frac{3}{p} \int_0^1 \abs{u(t)}^{3/p - 1} \cdot \abs{\dot{u}(t)} \idiff{t}  \\
	&\leq \frac{3}{2} (p-1)(p - 2) \cdot {M}_{3,p/3} \int_0^1 \sqrt{t} \idiff{t}
	= \frac{1}{2} (p-1)(p - 2) \cdot {M}_{3,p/3}.
\end{align*}
Since $u(1) = \Expect \abs{X}^{p}$ and $u(0) = \Expect \abs{Z}^{p}$,
this argument delivers the main result.
The other bounds follow from the numerical inequalities
\begin{equation} \label{eqn:cube-ineqs}
\abs{ a^3 - b^3 } \geq \abs{ a - b }^3
\quad\text{and}\quad
\abs{ a^3 - b^3 } \geq \abs{ a - b } \cdot (a^2 \vee b^2).
\end{equation}
The estimates~\eqref{eqn:cube-ineqs} are valid whenever $a, b \geq 0$.
\hfill\qed

\section{Calculus for rational matrix functions}
\label{sec:mtx-calculus}

To extend the scalar arguments to matrices, we need to introduce some
elements of the difference calculus for matrix functions.
There is nothing deep in this approach---it simply provides
a mechanism for the systematic calculation of matrix differences.
We will restrict our attention to rational functions,
as they suffice for our purposes and minimize technicalities.

The material in this section is adapted from
the monograph~\cite{KV14:Foundations-Free}
and the book~\cite[Sec.~3.2]{Hig08:Functions-Matrices},
where you may find proofs and further information.
These definitions are closely related to Voiculescu's noncommutative
derivative~\cite{Voi98:Analogues-Entropy-V}.

\subsection{Intuition: Nil-square infinitesimals}
\label{sec:infinitesimals}

An early approach to calculus involved \term{infinitesimals},
quantities that are smaller than every real number---yet nonzero.
While this perspective may be logically problematic,
it is both intuitive and useful~\cite{Tro99:Infinitesimals}.
This section sketches an algebraic formulation
that motivates a difference calculus for matrix functions.

Suppose that we want to compute the
derivative of the power function $f(a) \coloneqq a^p$
for a real argument $a \in \R$ and a natural number $p \in \N$.
We can conceive of the derivative as the relative 
change in the value of the function when its argument is
perturbed by an infinitesimal amount.  
To that end, let us moot the existence of a \term{nil-square infinitesimal},
an object $\iota$ with the property that $\iota \neq 0$
but $\iota^2 = 0$.
By a formal application of the binomial theorem and the nil-square property,
\[
f(a + \iota) - f(a) = (a + \iota)^p - a^p = p a^{p-1} \iota + {p \choose 2} a^{p-2} \iota^2
	+ \dots + \iota^{p} = p a^{p-1} \iota.
\]
The derivative $\Diff f(a) = p a^{p-1}$ agrees the coefficient of the infinitesimal $\iota$.
Similar heuristics are effective for any rational function.

We can realize the nil-square infinitesimal by working in the algebra
of upper-triangular $2 \times 2$ matrices.  Associate the real
number $a \in \R$ with the diagonal matrix $\diag(a, a)$,
and define the infinitesimal $\mtx{\iota}$ to be the Jordan block with eigenvalue zero:
\[
a \mapsto a \Id_2 = \begin{bmatrix} a & 0 \\ 0 & a \end{bmatrix}
\quad\text{and}\quad
\mtx{\iota} \coloneqq \begin{bmatrix} 0 & 1 \\ 0 & 0 \end{bmatrix}.
\]
Note that $\mtx{\iota}^2 = \mtx{0}$.
The same calculation from the last paragraph---now fully justified---yields
\[
f( a \Id_2 + \mtx{\iota} ) - f(a \Id_2)
	= \begin{bmatrix} a & 1 \\ 0 & a \end{bmatrix}^p
	- \begin{bmatrix} a & 0 \\ 0 & a \end{bmatrix}^p
	= \begin{bmatrix} 0 & p a^{p-1} \\ 0 & 0 \end{bmatrix}
	= p a^{p-1} \cdot \mtx{\iota}.
\]
The derivative $\Diff f(a) = p a^{p-1}$ appears as the coefficient of the infinitesimal $\mtx{\iota}$.

We can implement divided difference calculations with a related formalism.
Associate the pair $(a, b)$ of real numbers with the diagonal matrix $\diag(a, b)$.
Observe that
\[
\left( \begin{bmatrix} a & 0 \\ 0 & b \end{bmatrix} + \mtx{\iota} \right)^p
	- \begin{bmatrix} a & 0 \\ 0 & b \end{bmatrix}^p
	= \begin{bmatrix} 0 & \sum_{q=0}^{p-1} a^q b^{p-1-q} \\ 0 & 0 \end{bmatrix}
	= \left( \sum_{q=0}^{p-1} a^q b^{p-1-q} \right) \cdot \mtx{\iota}.
\]
The divided difference $\Delta f(a, b) = \sum_{q=0}^{p-1} a^q b^{p-1-q}$ appears as the coefficient
of the infinitesimal $\mtx{\iota}$.
Remarkably, these techniques extend to rational functions of square matrices.

\subsection{Rational matrix functions}

Introduce the class
$\coll{P}$ of univariate polynomials on $\C$, and define the class $\coll{R}$ of rational functions
acting on $\C$.  That is,
\[
\coll{R} \coloneqq \left\{ a \mapsto \frac{\phi(a)}{\psi(a)} : \text{$\phi, \psi \in \coll{P}$} \right\}.
\]
We can apply a rational function $f$ to a square matrix $\mtx{A} \in \M_d$ of any dimension $d$,
provided that the spectrum of $\mtx{A}$ contains no pole of $f$.  For example,
\[
f(a) \coloneqq \frac{a^2 + 1}{a - \iunit}
\quad\text{implies}\quad
f(\mtx{A}) = (\mtx{A}^2 + 1) (\mtx{A} - \iunit)^{-1},
\quad\text{provided $\iunit \notin \spec(\mtx{A})$.}
\]
As usual, we identify scalars with scalar matrices: $z \coloneqq z \Id$ for $z \in \C$.
Since rational functions of a single matrix commute, we can form the product in either
order, so $f(\mtx{A}) = (\mtx{A} - \iunit)^{-1} (\mtx{A}^2 + 1)$.

\subsection{Matrix differences}

Next, we introduce the \term{matrix difference operator} for rational matrix functions.
The simple and elegant definition is inspired by the calculations in
\cref{sec:infinitesimals}.

\begin{definition}[Matrix difference] \label{def:nc-diff}
Fix a rational function $f \in \coll{R}$.
Consider square matrices $\mtx{A}_0, \mtx{A}_1, \mtx{H} \in \M_d$ of the same dimension,
and assume that $\spec(\mtx{A}_0) \cup \spec(\mtx{A}_1)$ contains no pole of $f$.
The \term{matrix difference} of $f$ at the pair $(\mtx{A}_0, \mtx{A}_1)$
in the direction $\mtx{H}$ is the matrix
\[
\Delta f(\mtx{A}_0, \mtx{A}_1)[\mtx{H}] \coloneqq
	\northeast f\left( \begin{bmatrix} \mtx{A}_0 & \mtx{H} \\ \mtx{0} & \mtx{A}_1 \end{bmatrix} \right) 
	\in \M_d.
\]
The $\northeast$ map extracts the $(1, 2)$ block from the $2 \times 2$ block matrix.
\end{definition}

This definition allows us to compute the matrix difference for an arbitrary
rational function.  %
We offer two key examples.

\begin{example}[Power function: Matrix differences] \label{ex:power-mtx-diff}
Choose a natural number $p \in \N$, and consider the power function $f(a) \coloneqq a^p$ for $a \in \C$.
Fix arbitrary matrices $\mtx{A}, \mtx{B}, \mtx{H} \in \M_d$.  By repeated multiplication,
\[
\begin{bmatrix} \mtx{A} & \mtx{H} \\ \mtx{0} & \mtx{B} \end{bmatrix}^p
	= \begin{bmatrix} \mtx{A}^p & \sum_{q=0}^{p-1} \mtx{A}^q \mtx{H} \mtx{B}^{p-1-q} \\ \mtx{0} & \mtx{B}^p \end{bmatrix}.
\]
Therefore, \cref{def:nc-diff} yields
\[
\Delta f(\mtx{A}, \mtx{B})[\mtx{H}]
	= \sum_{q=0}^{p-1} \mtx{A}^q \mtx{H} \mtx{B}^{p-1-q}.
\]
This formula is consistent with the algebraic identity for a difference of powers:
\[
\mtx{A}^p - \mtx{B}^p = \sum_{q=0}^{p-1} \mtx{A}^q (\mtx{A} - \mtx{B}) \mtx{B}^{p-1-q}
	= \Delta f(\mtx{A}, \mtx{B})[\mtx{A} - \mtx{B}].
\]
Furthermore, the confluent matrix difference agrees with the ordinary directional derivative:
\[
\Diff f(\mtx{A})[\mtx{H}] \coloneqq \lim\nolimits_{t \downarrow 0} t^{-1} ((\mtx{A} + t\mtx{H})^p - \mtx{A}^p)
	= \sum_{q=0}^{p-1} \mtx{A}^q \mtx{H} \mtx{A}^{p-1-q}
	= \Delta f(\mtx{A}, \mtx{A})[ \mtx{H} ].
\]
As with the derivative, we interpret the matrix difference $\Delta f(\mtx{A}, \mtx{B})$
as a linear operator on the direction matrix $\mtx{H}$.
This operator generalizes the divided difference for scalars.
\end{example}

\begin{example}[Resolvents: Matrix differences]
Choose a complex number $\zeta \in \C \setminus \R$ with nonzero imaginary part,
and consider the resolvent function
$R_{\zeta}(a) \coloneqq (\zeta - a)^{-1}$ for $a \in \R$.
Fix \hilite{self-adjoint} matrices $\mtx{A}, \mtx{B}, \mtx{H} \in \Sym_d$,
which ensures that $\zeta \notin \spec(\mtx{A}) \cup \spec(\mtx{B})$.
By direct calculation, %
\[
R_{\zeta}\left( \begin{bmatrix} \mtx{A} & \mtx{H} \\ \mtx{0} & \mtx{B} \end{bmatrix} \right)
	= \begin{bmatrix} \zeta - \mtx{A} & -\mtx{H} \\ \mtx{0} & \zeta - \mtx{B} \end{bmatrix}^{-1}
	= \begin{bmatrix} (\zeta - \mtx{A})^{-1} & (\zeta - \mtx{A})^{-1} \mtx{H} (\zeta - \mtx{B})^{-1} \\ \mtx{0} & (\zeta - \mtx{B})^{-1} \end{bmatrix}.
\]
Therefore, the matrix difference of the resolvent function satisfies
\[
\Delta R_{\zeta}(\mtx{A}, \mtx{B})[\mtx{H}] = (\zeta - \mtx{A})^{-1} \mtx{H} (\zeta - \mtx{B})^{-1}
	= R_{\zeta}(\mtx{A}) \mtx{H} R_{\zeta}(\mtx{B}).
\]
This formula is consistent with the (second) resolvent identity:
\[
R_{\zeta}(\mtx{A}) - R_{\zeta}(\mtx{B}) %
	= R_{\zeta}(\mtx{A}) (\mtx{A} - \mtx{B}) R_{\zeta}(\mtx{B})
	= \Delta R_{\zeta}(\mtx{A}, \mtx{B})[\mtx{A} - \mtx{B}].
\]
Once again, the matrix difference $\Delta R_{\zeta}(\mtx{A}, \mtx{B})$ is a linear operator
on the direction $\mtx{H}$ that generalizes the scalar divided difference.
\end{example}

For the matrix difference operator $\Delta$,
some of the properties we have witnessed hold in general.
\begin{itemize}
\item	\textbf{Linearity.}  The function $\mtx{H} \mapsto \Delta f(\mtx{A}, \mtx{B})[ \mtx{H} ]$ is linear.

\item	\textbf{Differences.}  The difference of function values satisfies
\[
f(\mtx{A}) - f(\mtx{B}) = \Delta f(\mtx{A}, \mtx{B})[\mtx{A} - \mtx{B}].
\]

\item	\textbf{Derivatives.}  When the arguments of the matrix difference coalesce,
\[
\Delta f(\mtx{A}, \mtx{A})[\mtx{H}] = \Diff f(\mtx{A})[\mtx{H}].
\]
\end{itemize}
The matrix difference operator also satisfies a Leibniz rule for products
and a chain rule for composition.

\subsection{Matrix second differences}

By formal changes to \cref{def:nc-diff}, we can obtain second-order matrix differences.

\begin{definition}[Matrix second difference] \label{def:matrix-diff-2}
Fix a rational function $f \in \coll{R}$.  Consider square matrices
$\mtx{A}_0, \mtx{A}_1, \mtx{A}_2; \mtx{H}_1, \mtx{H}_2 \in \M_d$ of the same dimension,
and assume that $\spec(\mtx{A}_0) \cup \spec(\mtx{A}_1) \cup \spec(\mtx{A}_2)$
contains no pole of $f$.  The \term{second-order matrix difference}
of $f$ at the triple $(\mtx{A}_0, \mtx{A}_1, \mtx{A}_2)$ in the direction
$(\mtx{H}_1, \mtx{H}_2)$ is the matrix
\[
\Delta^2 f( \mtx{A}_0, \mtx{A}_1, \mtx{A}_2 )[ \mtx{H}_1 \otimes \mtx{H}_2 ]
	\coloneqq \northeast f\left( \begin{bmatrix}
	\mtx{A}_0 & \mtx{H}_1 & \\
	& \mtx{A}_1 & \mtx{H}_2  \\
	&& \mtx{A}_2
\end{bmatrix} \right).
\]
The matrix is block bidiagonal, and the unexpressed blocks are zero matrices.
The $\northeast$ map extracts the $(1, 3)$ block from the $3 \times 3$ block matrix. 
\end{definition}

We have written the direction argument $\mtx{H}_1 \otimes \mtx{H}_2$ as a tensor
product to emphasize the fact that the second-order matrix difference $\Delta^2 f(\mtx{A}_0, \mtx{A}_1, \mtx{A}_2)$
is a bilinear function of the argument $(\mtx{H}_1, \mtx{H}_2)$.  By linearity, the second-order matrix
difference extends to the linear space spanned by these tensors.

\begin{example}[Power function: Matrix second difference] \label{ex:power-second-diff}
Let us continue with \cref{ex:power-mtx-diff}.
Fix arbitrary matrices $\mtx{A}, \mtx{B}, \mtx{C}; \mtx{G}, \mtx{H} \in \M_d$.
For a natural number $p \geq 2$, by repeated multiplication,
\[
\begin{bmatrix} \mtx{A} & \mtx{G} & \\
& \mtx{B} & \mtx{H} \\
&& \mtx{C}
\end{bmatrix}^p
= \begin{bmatrix}
	\mtx{A}^p & \sum_{q+r = p-1} \mtx{A}^q \mtx{G} \mtx{B}^r & \sum_{q+r+s = p-2} \mtx{A}^q \mtx{G} \mtx{B}^r \mtx{H} \mtx{C}^s \\
	& \mtx{B}^p & \sum_{r+s = p - 1} \mtx{B}^r \mtx{H} \mtx{C}^s \\
	&& \mtx{C}^p
	\end{bmatrix}.
\]
The indices $q,r,s \in \Z_+$ range over nonnegative integers.	
\Cref{def:matrix-diff-2} delivers the second-order matrix difference
of the power function $f(a) \coloneqq a^p$:
\[
\Delta^2 f(\mtx{A}, \mtx{B}, \mtx{C})[ \mtx{G} \otimes \mtx{H} ]
	= \sum_{q+r+s=p-2} \mtx{A}^q \mtx{G} \mtx{B}^r \mtx{H} \mtx{C}^s.
\]
It is not a coincidence that the power of the block matrix also
displays the first-order matrix differences.
Specializing to the confluent case,
\[
\Delta^2 f(\mtx{A}, \mtx{A}, \mtx{A})[ \mtx{G} \otimes \mtx{H} ]
	= \sum_{q+r+s=p-2} \mtx{A}^q \mtx{G} \mtx{A}^r \mtx{H} \mtx{A}^s
	= \frac{1}{2} \Diff^2 f(\mtx{A})[ \mtx{G} \otimes \mtx{H} ].
\]
So the matrix second difference also contains the second directional derivative.
\end{example}

\begin{example}[Inverse powers: Matrix second difference] \label{ex:inv-power-second-diff}
For invertible $\mtx{A}, \mtx{B}, \mtx{C} \in \M_d$, we can compute the inverse of the block bidiagonal matrix via triangular substitution:
\[
\begin{bmatrix} \mtx{A} & \mtx{G} & \\
& \mtx{B} & \mtx{H} \\
&& \mtx{C}
\end{bmatrix}^{-1}
= \begin{bmatrix}
	\mtx{A}^{-1} & -\mtx{A}^{-1} \mtx{G} \mtx{B}^{-1} & \mtx{A}^{-1} \mtx{G} \mtx{B}^{-1} \mtx{H} \mtx{C}^{-1} \\
	& \mtx{B}^{-1} & -\mtx{B}^{-1} \mtx{H} \mtx{C}^{-1} \\
	&& \mtx{C}^{-1}
\end{bmatrix}
\]
This leads to a quick derivation of the matrix second
difference of the inverse function $f(a) \coloneqq a^{-1}$.
Indeed,
\[
\Delta^2 f(\mtx{A}, \mtx{B}, \mtx{C})[ \mtx{G} \otimes \mtx{H} ]
	= \mtx{A}^{-1} \mtx{G} \mtx{B}^{-1} \mtx{H} \mtx{C}^{-1}.
\]
In the confluent case, we obtain half the second derivative of
the inverse function.

From here, we can easily obtain the second differences of inverse powers
by repeated multiplication of the inverse with itself.  Indeed, for $p \in \N$,
\[
\begin{bmatrix} \mtx{A} & \mtx{G} & \\
& \mtx{B} & \mtx{H} \\
&& \mtx{C}
\end{bmatrix}^{-p}
= \begin{bmatrix}
	\mtx{A}^{-p} & -\sum_{q+r=p+1}^+ \mtx{A}^{-q} \mtx{G} \mtx{B}^{-r} & \sum_{q+r+s = p+2}^+ \mtx{A}^{-q} \mtx{G} \mtx{B}^{-r} \mtx{H} \mtx{C}^{-s} \\
	& \mtx{B}^{-p} & - \sum_{r+s=p+1}^+ \mtx{B}^{-r} \mtx{H} \mtx{C}^{-s} \\
	&& \mtx{C}^{-p}
\end{bmatrix}
\]
The superscript $+$ attached to the summation sign means that
the indices $q,r,s \in \N$ extend over the \hilite{natural numbers}.
Thus, for $g(a) \coloneqq a^{-p}$, we have
\[
\Delta^2 g(\mtx{A}, \mtx{B}, \mtx{C})[ \mtx{G} \otimes \mtx{H} ]
	= \sum_{q+r+s = p+2}^+ \mtx{A}^{-q} \mtx{G} \mtx{B}^{-r} \mtx{H} \mtx{C}^{-s}
\]
This calculation will play a role in the study of resolvent powers.
\end{example}

For the matrix second difference operator $\Delta^2$,
the following properties hold in general.
\begin{itemize}
\item	\textbf{Bilinearity.}  The map $(\mtx{H}_1, \mtx{H}_2) \mapsto \Delta^2 f(\mtx{A}_0, \mtx{A}_1, \mtx{A}_2)[\mtx{H}_1 \otimes \mtx{H}_2]$ is bilinear.

\item	\textbf{Differences.}  The difference of matrix differences satisfies
\begin{align*}
\Delta f(\mtx{A}, \mtx{B})[\mtx{H}] - \Delta f(\mtx{A}, \mtx{C})[\mtx{H}]
	&= \Delta^{2} f(\mtx{A}, \mtx{B}, \mtx{C})[ \mtx{H} \otimes  (\mtx{B}-\mtx{C}) ]; \\
\Delta f(\mtx{A}, \mtx{B})[\mtx{H}] - \Delta f(\mtx{B}, \mtx{B})[\mtx{H}]
	&= \Delta^{2} f(\mtx{A}, \mtx{B}, \mtx{B})[ (\mtx{A}-\mtx{B}) \otimes \mtx{H} ]. %
\end{align*}

\item	\textbf{Derivatives.}  When the arguments of the matrix difference coalesce,
\[
\Delta^2 f(\mtx{A}, \mtx{A}, \mtx{A}) = \frac{1}{2} \Diff^2 f(\mtx{A}).
\]
\end{itemize}
For further details and extensions, see the survey~\cite{KV14:Foundations-Free}
on free noncommutative function theory.

\section{Matrix consolidation inequalities}
\label{sec:mtx-solid}

The scalar universality arguments rely on numerical inequalities
for divided differences, such as \cref{lem:indep-sum-cvx-bd}.
In the matrix case, the situation is more complicated
because many plausible analogs of this lemma fail.
Nevertheless, we can establish satisfactory estimates
after some elementary arguments.

\subsection{Matrix consolidation}

We begin with a deterministic trace inequality
that allows us to combine some of the factors in a matrix product.

\begin{proposition}[Matrix consolidation] \label{prop:mtx-solid}
Consider \hilite{self-adjoint} matrices $\mtx{H}_j \in \Sym_d$ and \hilite{normal} matrices $\mtx{A}_j \in \M_d$ for $j =0,1,2$.
For all positive integers $q, r, s\in\Z_+$ with $p \coloneqq q + r + s$, it holds that
\[
\labs{\ntr\big[ \mtx{H}_0 \mtx{A}_0^{q} \mtx{H}_1 \mtx{A}_1^{r} \mtx{H}_2 \mtx{A}_2^{s} \big] }
	\leq \sum_{j,k,\ell=0}^2 \alpha_{jk\ell} \ntr\big[ \norm{ \mtx{H}_j } \,\mtx{H}_k^2 \cdot \abs{\mtx{A}_\ell}^p \big].
\]
The scalar coefficients $\alpha_{jk\ell}$ are nonnegative and sum to one; they depend on the exponents $q,r,s$. 
\end{proposition}

The proof relies on a more basic trace inequality:

\begin{fact}[Matrix GM--AM inequality] \label{fact:matrix-gm-am}
Consider psd matrices $\mtx{A}, \mtx{B} \in \Sym_d$ and self-adjoint $\mtx{H} \in \Sym_d$.
For each $\theta \in [0,1]$,
\[
\labs{ \ntr\big[ \mtx{H} \mtx{A}^\theta \mtx{H} \mtx{B}^{1-\theta} \big] }
	\leq \theta \ntr[  \mtx{H}^2 \, \mtx{A} ] + (1-\theta) [ \mtx{H}^2 \,  \mtx{B} ].
\]
\end{fact}

\begin{proof}[Proof sketch]
Introduce the spectral resolutions of $\mtx{A}, \mtx{B}$.
Apply the scalar GM--AM inequality to the eigenvalues of $\mtx{A}, \mtx{B}$.
Reassemble the matrices.  For details, see~\cite[Fact~2.2]{Tro16:Expected-Norm}. %
\end{proof}

\begin{proof}[Proof of \cref{prop:mtx-solid}]
By cycling the trace, we may assume that $r$ is the maximum exponent.  The spectral norm bound for the trace yields
\begin{align*}
E &\coloneqq \labs{\ntr\big[ \mtx{H}_0 \mtx{A}_0^{q} \mtx{H}_1 \mtx{A}_1^{r} \mtx{H}_2 \mtx{A}_2^{s} \big] }
	\leq \norm{ \mtx{H}_0 } \cdot \ntr \labs{ \mtx{A}_0^{q} \mtx{H}_1 \mtx{A}_1^{r} \mtx{H}_2 \mtx{A}_2^{s} }.
\end{align*}
Since $r$ is the maximum exponent,
the other exponents $q, s \leq p/2$
and $r = (p/2 - q) + (p/2 - s)$.
Introducing the polar factorization of the normal matrix $\mtx{A}_1 = \mtx{U} \abs{\mtx{A}_1} = \abs{\mtx{A}_1} \mtx{U}$,
we can expand the $r$th power $\mtx{A}_1^r$ as
\begin{align*}
E &= \norm{ \mtx{H}_0 } \cdot \ntr \big\vert \mtx{A}_0^{q} \mtx{H}_1 \abs{\mtx{A}_1}^{p/2-q} \mtx{U}^r \cdot \abs{\mtx{A}_1}^{p/2-s} \mtx{H}_2 \mtx{A}_2^{s} \big\vert.
\end{align*}
Apply the Cauchy--Schwarz inequality for the Schatten 1-norm,
and cycle the resulting traces to consolidate powers of $q$ and $s$.
Thus,
\begin{align*}
E &\leq \norm{\mtx{H}_0} \cdot \left( \ntr\big[ \mtx{H}_1 \abs{\mtx{A}_0}^{2q} \mtx{H}_1 \abs{\mtx{A}_1}^{p-2q} \big] \right)^{1/2} \cdot
	\left( \ntr\big[ \mtx{H}_2 \abs{\mtx{A}_1}^{p-2s} \mtx{H}_2 \abs{\mtx{A}_2}^{2s}  \big] \right)^{1/2} \\
	&\leq
	\frac{1}{2} \norm{ \mtx{H}_0 } \cdot \left( \ntr\big[ \mtx{H}_1 \abs{\mtx{A}_0}^{2q} \mtx{H}_1 \abs{\mtx{A}_1}^{p - 2q} \big]
	+ \ntr\big[ \mtx{H}_2 \abs{\mtx{A}_1}^{p - 2s} \mtx{H}_2 \abs{\mtx{A}_2}^{2s} \big] \right).
\end{align*}
To arrive at the second line, we invoked the scalar GM--AM inequality.
Apply the matrix GM--AM inequality to each trace to arrive at
\begin{align*}
E \leq \frac{1}{2} \norm{ \mtx{H}_0 } \cdot \bigg( &\frac{2q}{p} \cdot \ntr\left[  \mtx{H}_1^2 \, \abs{ \mtx{A}_0 }^p \right]
	+ \frac{p - 2q}{p} \cdot \ntr\left[  \mtx{H}_1^2  \,\abs{ \mtx{A}_1 }^p  \right] \\
	&+ \frac{p - 2s}{p} \cdot \ntr\left[  \mtx{H}_2^2  \,\abs{ \mtx{A}_1 }^{p} \right]
	+ \frac{2s}{p} \cdot \ntr\left[ \mtx{H}_2^2  \,\abs{ \mtx{A}_2 }^p \right] \bigg).
\end{align*}
Last, observe that the scalar coefficients compose a convex combination.
\end{proof}

\begin{remark}[Consolidation]
One might hope to prove stronger inequalities of the form
\[
\labs{ \ntr\big[ \mtx{HA}^q \mtx{HA}^r \mtx{HA}^s \big] }
	\overset{?}{\leq} \ntr\big[ \abs{\mtx{H}}^3 \, \abs{\mtx{A}}^p \big].
\]
In general, this putative statement is false~\cite[Sec.~IV]{CL22:Trace-Inequality},
even when $\mtx{H}, \mtx{A}$ are positive semidefinite and $s = 0$.
\cref{prop:mtx-solid} offers a passable substitute.
We can obtain variants by using more sophisticated
trace inequalities.
\end{remark}

\subsection{Matrix consolidation with exchangeable arguments}

Next, we upgrade \cref{prop:mtx-solid} to a trace inequality
for exchangeable random matrices.  This result serves as a
replacement for \cref{lem:indep-sum-cvx-bd} when we treat
matrix differences of power functions.

\begin{proposition}[Matrix consolidation: Exchangeable case] \label{prop:mtx-solid-exch}
Construct an exchangeable triple $(\mtx{X}, \mtx{X}', \mtx{X}'')$ of random matrices.
Consider a matrix function $\mtx{Y} \coloneqq \mtx{Y}(\mtx{X}) \in \M_d$
that takes normal values,
and form exchangeable counterparts $\mtx{Y}' \coloneqq \mtx{Y}(\mtx{X}')$
and $\mtx{Y}'' \coloneqq \mtx{Y}(\mtx{X}'')$.
Consider matrix functions $\mtx{W}_j \coloneqq \mtx{W}_j(\mtx{X}, \mtx{X}', \mtx{X}'') \in \Sym_d$
taking self-adjoint values.
For all positive integers $q, r, s \in \Z_+$ with $p \coloneqq q + r + s$,
\[
\labs{ \Expect \ntr\big[ \mtx{W}_0\, (\mtx{Y})^q\, \mtx{W}_1\, (\mtx{Y}')^r\, \mtx{W}_2 \,(\mtx{Y}'')^s \big] }
	\leq \max\nolimits_{\pi, j, k} \Expect \ntr\big[ \Expect\big[ \norm{ \mtx{W}_j^\pi } \, (\mtx{W}_k^\pi)^2 \lcondbar \mtx{X} \big] \cdot \abs{\mtx{Y}}^p \big].
\]
The random matrix $\mtx{W}_j^\pi(\mtx{X}, \mtx{X}', \mtx{X}'')$ is obtained by permuting
the arguments of $\mtx{W}_j$ with $\pi$.
\end{proposition}

\Cref{prop:mtx-solid-exch} has a similar structure to the existing trace
inequality~\cite[Prop.~5.1]{BvH24:Universality-Sharp},
but it does not follow from that result.

\begin{proof}
We employ the notation $\mtx{Y}^{(\ell)}$ for the exchangeable counterparts, where $\ell$ counts the number of prime symbols.
The quantity of interest takes the form
\[
E \coloneqq \labs{ \Expect \ntr\big[ \mtx{W}_0 (\mtx{Y}^{(0)})^q \mtx{W}_1 (\mtx{Y}^{(1)})^r \mtx{W}_2 (\mtx{Y}^{(2)})^s \big] }.
\]
\Cref{prop:mtx-solid} immediately yields
\begin{align*}
E \leq \Expect \left[ \sum_{j,k,\ell} \alpha_{jk\ell} \ntr\big[ \norm{\mtx{W}_j} \, \mtx{W}_k^2 \cdot \abs{ \mtx{Y}^{(\ell)} }^p \big] \right] %
	\leq \max\nolimits_{j,k,\ell} \Expect \ntr \big[ \norm{ \mtx{W}_j } \, \mtx{W}_k^2 \cdot \abs{ \mtx{Y}^{(\ell)} }^p \big].
\end{align*}
We bounded the convex combination above by the maximum.
Invoke exchangeability to swap $\mtx{X}^{(\ell)} \leftrightarrow \mtx{X}$,
which replaces $\mtx{Y}^{(\ell)}$ with $\mtx{Y}$ and permutes the arguments of the $\mtx{W}_j$.
Thus,
\begin{align*}
E \leq \max\nolimits_{\pi, j,k } \Expect \ntr \big[ \norm{ \mtx{W}_j^\pi } \,(\mtx{W}_k^\pi)^2 \cdot \abs{ \mtx{Y} }^p \big]
	= \max\nolimits_{\pi,j,k}  \Expect \ntr \big[ \Expect\big[ \norm{ \mtx{W}_j^\pi } \,(\mtx{W}_k^\pi)^2 \lcondbar \mtx{X} \big] \cdot \abs{ \mtx{Y} }^p \big]
\end{align*}
We have conditioned on $\mtx{X}$ to average out the extra
randomness in $\mtx{X}', \mtx{X}''$.
\end{proof}

\section{Covariance identities: Matrix setting}
\label{sec:cov-matrix}

In this section, we begin the development of our universality
results for random matrices.  The first step is to reproduce
the covariance identities from \cref{sec:cov-scalar}.  The
arguments are essentially the same, but the notation is more
involved because of noncommutativity.

\subsection{Preliminaries}

The derivative of the trace of a spectral function has a particularly simple form.
Fix a rational function $f \in \coll{R}$, and
write $g \coloneqq \Diff f$ for the (scalar) derivative.  For
matrices $\mtx{A}, \mtx{H} \in \M_d$ where $\spec(\mtx{A})$
does not contain a pole of $f$, %
\begin{equation} \label{eqn:D-tr}
(\Diff \,\ntr f)(\mtx{A})[\mtx{H}] = \ntr[ \mtx{H} \cdot (\Diff f)(\mtx{A}) ]
	= \ntr[ \mtx{H} g(\mtx{A}) ]
\end{equation}
We will exploit the expression~\eqref{eqn:D-tr} to avoid taking
third-order matrix differences.

Let $\mtx{F} : \Sym_d \to \M_d$ be a Fr{\'e}chet-differentiable matrix function.
We introduce a bilinear form associated with the derivative.  For self-adjoint
$\mtx{A}, \mtx{H}_1, \mtx{H}_2 \in \Sym_d$, 
\begin{equation} \label{eqn:mtx-pairing}
\ip{ \Diff \mtx{F}(\mtx{A}) }{ \mtx{H}_1 \otimes \mtx{H}_2 }
	\coloneqq \ntr[ \mtx{H}_1 \cdot \Diff \mtx{F}(\mtx{A})[\mtx{H}_2] ].
\end{equation}
The pairing is linear in each of the two arguments.

\subsection{Gaussian IBP}

By conditioning, the Gaussian IBP rule (\cref{fact:gauss-ibp}) extends to multivariate Gaussian distributions
and, in particular, to Gaussian random matrices.

\begin{proposition}[Gaussian IBP] \label{prop:gauss-ibp-mtx}
Fix a Fr{\'e}chet differentiable function $\mtx{F} : \Sym_d \to \M_d$.
Let $\mtx{Z} \in \Sym_d$ be a Gaussian self-adjoint matrix.
Then
\[
\Expect \ntr[ (\mtx{Z} - \Expect \mtx{Z}) \mtx{F}(\mtx{Z}) ]
	= \ip{ \Expect \Diff \mtx{F}(\mtx{Z}) }{ \Varo[ \mtx{Z} ] }.
\]
The pairing %
is defined in~\eqref{eqn:mtx-pairing},
and the variance tensor $\Varo[\mtx{Z}]$ is defined in~\eqref{eqn:Varo}. 
\end{proposition}

\begin{proof}
We can express the Gaussian matrix $\mtx{Z}$ in terms of an independent family $(\gamma_1, \dots, \gamma_n)$
of standard normal variables:
\[
\mtx{Z} = \mtx{A}_0 + \sum_{i=1}^n \gamma_i \mtx{A}_i
\quad\text{for fixed $\mtx{A}_i \in \Sym_d$.}
\]
Observe that the variance tensor satisfies
\[
\Varo[\mtx{Z}] \coloneqq \Expect[ (\mtx{Z} - \Expect\mtx{Z}) \otimes (\mtx{Z} - \Expect\mtx{Z}) ] = \sum_{i=1}^n \mtx{A}_i \otimes \mtx{A}_i.
\]
By linearity of the derivative and the scalar Gaussian IBP rule (\cref{fact:gauss-ibp}),
\begin{align*}
\Expect \ntr[ (\mtx{Z} - \Expect\mtx{Z}) \mtx{F}(\mtx{Z}) ]
	&= \sum_{i=1}^n \Expect[ \gamma_i \cdot \ntr[ \mtx{A}_i \mtx{F}(\mtx{Z}) ] ]
	= \sum_{i=1}^n \Expect \ntr[ \mtx{A}_i \cdot \Diff \mtx{F}(\mtx{Z})[\mtx{A}_i] ] \\
	&= \sum_{i=1}^n \Expect \ip{ \Diff \mtx{F}(\mtx{Z}) }{ \mtx{A}_i \otimes \mtx{A}_i }
	= \ip{ \Expect \Diff \mtx{F}(\mtx{Z}) }{ \Varo[\mtx{Z}] }.
\end{align*}
We used the chain rule and the fact that $\Diff_{\gamma_i} \mtx{Z} = \mtx{A}_i$ to obtain the
directional derivative.  The last step follows from bilinearity of the pairing.
\end{proof}

\subsection{Discrete covariance identities}
\label{sec:mtx-cov-id}

In \cref{sec:indep-sum-exch}, we introduced the independent sum of scalar random variables,
and we designed an exchangeable counterpart.  This whole discussion extends to the
matrix setting without any essential change.  Let us summarize.

Consider an independent family $(\mtx{S}_1, \dots, \mtx{S}_n)$ of
random \hilite{self-adjoint} matrices taking values in $\Sym_d$.
Introduce the random matrix
\begin{equation} \label{eqn:indep-sum-mtx}
\mtx{X} \coloneqq \sum_{i=1}^n \mtx{S}_i \in \Sym_d. %
\end{equation}
Draw two independent copies $(\mtx{S}_1', \dots, \mtx{S}_n')$ and
$(\mtx{S}_1'', \dots, \mtx{S}_n'')$ of the summands
and an independent random variable $I \sim \uniform\{1, \dots, n\}$.
The exchangeable counterparts for $\mtx{X}$ take the form
\begin{equation} \label{eqn:indep-sum-mtx-exch}
\begin{aligned}
\mtx{X}' &\coloneqq \mtx{X} + (\mtx{S}_I' - \mtx{S}_I) = \mtx{S}_I' + \sum_{j \neq I} \mtx{S}_j; \\
\mtx{X}'' &\coloneqq \mtx{X} + (\mtx{S}_I'' - \mtx{S}_I) = \mtx{S}_I'' + \sum_{j \neq I} \mtx{S}_j.
\end{aligned}
\end{equation}
As before, we use the same value of $I$ in both constructions.

The exchangeable pair $(\mtx{X}, \mtx{X}')$ satisfies the linear regression condition:
\[
\Expect[ \mtx{X} - \mtx{X}' \condbar \mtx{X} ] = n^{-1} (\mtx{X} - \Expect \mtx{X}).
\]
From here, we can obtain covariance identities in the spirit of
\cref{prop:discrete-ibp-sum}.
We can compute the ``covariance'' between the random matrix
and a function $\mtx{F} : \Sym_d \to \M_d$ of the random matrix:
\begin{equation} \label{eqn:mtx-cov-id}
\Expect \ntr[ (\mtx{X} - \Expect \mtx{X}) \mtx{F}(\mtx{X}) ]
	= \frac{n}{2} \Expect \ntr[ (\mtx{X} - \mtx{X}') (\mtx{F}(\mtx{X}) - \mtx{F}(\mtx{X}')) ].
\end{equation}
Similarly, we can express the variance tensor as
\begin{equation} \label{eqn:var-tensor-exch}
\Varo[ \mtx{X} ] \coloneqq \Expect[ (\mtx{X} - \Expect \mtx{X}) \otimes (\mtx{X} - \Expect \mtx{X}) ]
	= \frac{n}{2} \Expect [ (\mtx{X} - \mtx{X}') \otimes (\mtx{X} - \mtx{X}') ].
\end{equation}
The proofs of~\eqref{eqn:var-tensor-exch} and~\eqref{eqn:mtx-cov-id}
are identical in structure to the proof of \cref{prop:discrete-ibp-sum}.

\subsection{Discrete IBP}

Next, we establish a matrix IBP rule for an independent sum
that expresses the matrix covariance %
in terms of derivatives and divided differences.
This result and its proof are analogous with the discrete IBP 
rule for scalars (\cref{thm:discrete-ibp-sum}).
To streamline future applications, we make a slight generalization
to allow affine functions of the independent sum.

\begin{theorem}[Discrete IBP: Matrix series] \label{thm:discrete-ibp-mtx}
As in \cref{sec:mtx-cov-id}, consider the exchangeable triple $(\mtx{X}, \mtx{X}', \mtx{X}'')$
induced by an independent sum of $n$ random \hilite{self-adjoint} matrices taking values in $\Sym_d$.
Consider an affine function
\[
\mtx{Y} \coloneqq \mtx{Y}(\mtx{X}) \coloneqq \alpha \mtx{X} + \mtx{A}
\quad\text{with $\alpha \in \R$ and $\mtx{A} \in \Sym_d$}
\]
with exchangeable counterparts $\mtx{Y}' \coloneqq \mtx{Y}(\mtx{X}')$ and $\mtx{Y}'' \coloneqq \mtx{Y}(\mtx{X}'')$.
For each rational function $f \in \coll{R}$ with no real poles,
\begin{align*}
\Expect\ntr[ (\mtx{X} - \Expect \mtx{X}) f(\mtx{Y}) ]
	&= \alpha \cdot \ip{ \Expect \Diff f(\mtx{Y}) }{ \Varo[\mtx{X}] } \\
	&\quad+ \frac{\alpha^2 n}{2} \Expect \ntr[ (\mtx{X} - \mtx{X}') \cdot \Delta^2 f(\mtx{Y}, \mtx{Y}', \mtx{Y}'')[ (\mtx{X} - \mtx{X}') \otimes (\mtx{X}' - \mtx{X}'') ] ] \\
	&\quad+ \frac{\alpha^2 n}{2} \Expect \ntr[ (\mtx{X} - \mtx{X}') \cdot \Delta^2 f(\mtx{Y}, \mtx{Y}'', \mtx{Y}'')[ (\mtx{X} - \mtx{X}'') \otimes (\mtx{X} - \mtx{X}') ] ].
\end{align*}
\end{theorem}

\begin{proof}
We begin with the matrix covariance identity~\eqref{eqn:mtx-cov-id},
and we express the difference of functions using a matrix difference
operator:
\begin{align*}
E \coloneqq \Expect \ntr[ (\mtx{X} - \Expect\mtx{X}) f(\mtx{Y}) ]
	&= \frac{n}{2} \Expect \ntr[ (\mtx{X} - \mtx{X}') (f(\mtx{Y}) - f(\mtx{Y}')) ] \\
	&= \frac{\alpha n}{2} \Expect \ntr[ (\mtx{X} - \mtx{X}') \cdot \Delta f(\mtx{Y}, \mtx{Y}')[ \mtx{X} - \mtx{X}'] \big].
\end{align*}
Indeed, the difference $\mtx{Y} - \mtx{Y}' = \alpha (\mtx{X} - \mtx{X}')$.
To the matrix difference, add and subtract $\Delta f(\mtx{Y}'', \mtx{Y}'') = \Diff f(\mtx{Y}'')$ to reach
\begin{align*}
E &= \frac{\alpha n}{2} \Expect \ntr[ (\mtx{X} - \mtx{X}') \cdot \Diff f(\mtx{Y}'')[\mtx{X} - \mtx{X}'] ] \\
	&+ \frac{\alpha n}{2} \Expect \ntr\big[ (\mtx{X} - \mtx{X}') \cdot \big( \Delta f(\mtx{Y},\mtx{Y}') - \Delta f(\mtx{Y}'', \mtx{Y}'') \big)[\mtx{X} - \mtx{X}'] \big]
	\eqqcolon \onecirc + \twocirc.
\end{align*}
Let us treat the two summands in turn.

Write the first summand in terms of the pairing:
\[
\onecirc %
	= \frac{\alpha n}{2} \Expect \ip{ \Diff f(\mtx{Y}'') }{ (\mtx{X} - \mtx{X}') \otimes (\mtx{X} - \mtx{X}') }.
\]
Since the random matrix $\mtx{Y}'' = \alpha \mtx{X}'' + \mtx{A}$ is independent from $\mtx{X} - \mtx{X}'$, we can compute the expectation of the two arguments of the pairing independently:
\begin{align*}
\onecirc %
	= \alpha \cdot \lip{ \Expect \Diff f(\mtx{Y}'') }{ \frac{n}{2} \Expect[ (\mtx{X} - \mtx{X}') \otimes (\mtx{X} - \mtx{X}') ] }
	= \alpha \cdot \ip{ \Expect \Diff f(\mtx{Y}) }{ \Varo[ \mtx{X} ] }.
\end{align*}
In the first argument of the pairing, we have made the exchange $\mtx{Y}'' \leftrightarrow \mtx{Y}$.
For the second argument, we exploit the relation~\eqref{eqn:var-tensor-exch} for the variance
tensor.

For the second summand, we reduce the first differences to second differences.
To that end, add and subtract $\Delta f(\mtx{Y}, \mtx{Y}'')$ to reach
\begin{align*}
\twocirc
	&=\frac{\alpha n}{2} \Expect \ntr\big[ (\mtx{X} - \mtx{X}') \cdot \big( \Delta f(\mtx{Y},\mtx{Y}') - \Delta f(\mtx{Y}, \mtx{Y}'') + \Delta f(\mtx{Y}, \mtx{Y}'') - \Delta f(\mtx{Y}'', \mtx{Y}'') \big)[\mtx{X} - \mtx{X}'] \big] \\
	&= \frac{\alpha^2 n}{2} \Expect \ntr[ (\mtx{X} - \mtx{X}') \cdot \Delta^2 f(\mtx{Y}, \mtx{Y}', \mtx{Y}'')[ (\mtx{X} - \mtx{X}') \otimes (\mtx{X}' - \mtx{X}'') ] ]\\
	&+ \frac{\alpha^2 n}{2} \Expect \ntr[ (\mtx{X} - \mtx{X}') \cdot \Delta^2 f(\mtx{Y}, \mtx{Y}'', \mtx{Y}'')[ (\mtx{X} - \mtx{X}'') \otimes (\mtx{X} - \mtx{X}') ] ].
\end{align*}
Combine the expressions for $\onecirc$ and $\twocirc$ to obtain the
stated formula for $E$.
\end{proof}

\section{Interpolation: Matrix setting}
\label{sec:interp-matrix}

In this section, we execute the basic interpolation arguments
in the matrix setting.  There is no conceptual difference with
the scalar case (\cref{prop:interp-sum}).

\subsection{Interpolation}

Consider a random self-adjoint matrix $\mtx{X} \in \Sym_d$.
Independent from $\mtx{X}$, construct a Gaussian self-adjoint matrix
$\mtx{Z} \in \Sym_d$ with the same expectation and variance tensor:
\begin{equation} \label{eqn:gauss-proxy-pf}
\Expect[ \mtx{Z} ] = \Expect[ \mtx{X} ] %
\quad\text{and}\quad
\Varo[\mtx{Z}] = \Varo[\mtx{X}]. %
\end{equation}
Construct a family of interpolants:
\begin{equation} \label{eqn:mtx-interp}
\mtx{Y}_t \coloneqq \mtx{Y}_t(\mtx{X}, \mtx{Z}) \coloneqq
	\Expect[\mtx{X}] + \sqrt{t} \cdot (\mtx{X} - \Expect\mtx{X}) + \sqrt{1 - t} \cdot (\mtx{Z}-\Expect\mtx{Z})
	\quad\text{for all $t \in [0, 1]$.}
\end{equation}
Observe that $\mtx{Y}_1 = \mtx{X}$ and $\mtx{Y}_0 = \mtx{Z}$.
Furthermore, the expectation and variance tensor remain constant:
\[
\Expect[ \mtx{Y}_t ] = \Expect[ \mtx{X} ]
\quad\text{and}\quad
\Varo[ \mtx{Y}_t ] = \Varo[\mtx{X}]
\quad\text{for all $t \in [0,1]$.}
\]
We intend to bound derivatives of functions along the interpolation path.

\begin{proposition}[Matrix interpolation] \label{prop:mtx-interp}
As above, introduce random self-adjoint matrices $\mtx{X}, \mtx{Z} \in \Sym_d$
along with the interpolants $\mtx{Y}_t$ from~\eqref{eqn:mtx-interp}.
Consider a rational function $h \in \coll{R}$ with no poles on the
real line, and write $f \coloneqq \Diff h$.  Define the interpolants
\[
u(t) \coloneqq \Expect \ntr h(\mtx{Y}_t)
\quad\text{for all $t \in [0,1]$.}
\]
For $t \in (0,1)$, the time derivative $\dot{u}(t)$ takes the form
\begin{equation} \label{eqn:mtx-mom-interp}
\dot{u}(t) = \frac{1}{2\sqrt{t}} \Expect \ntr[ (\mtx{X}-\Expect \mtx{X}) f(\mtx{Y}_t(\mtx{X}, \mtx{Z})) ]
	- \frac{1}{2\sqrt{1 - t}} \Expect \ntr[ (\mtx{Z}-\Expect \mtx{Z}) f(\mtx{Y}_t(\mtx{X}, \mtx{Z})) ].
\end{equation}
\end{proposition}

\begin{proof}
Use dominated convergence to pass the derivative through the expectation.
Apply the formula~\eqref{eqn:D-tr} for the derivative of a trace function,
along with the chain rule.
\end{proof}

\subsection{Interpolation: Matrix sum}

\Cref{prop:mtx-interp} shows that the derivative $\dot{u}$ of the moment
interpolant contains two covariance terms.  For an independent sum of
random matrices, we can use our integration by parts rules to compare
these two quantities.

\begin{proposition}[Interpolation: Matrix sum] \label{prop:mtx-interp-sum}
As in \cref{thm:discrete-ibp-mtx},
consider the exchangeable triple $(\mtx{X}, \mtx{X}', \mtx{X}'')$
induced by an independent sum of random self-adjoint matrices with $n$ terms.
Construct an independent Gaussian matrix $\mtx{Z} \sim \normal(\Expect[\mtx{X}], \Varo[\mtx{X}])$. %

Let $h \in \coll{R}$ be a rational function with no poles on the real line,
and define the derivative $f \coloneqq \Diff h$.
For $t \in [0,1]$, introduce the interpolants
$u(t) \coloneqq \Expect \ntr h(\mtx{Y}_t)$ where $\mtx{Y}_t$ is
defined in~\eqref{eqn:mtx-interp}.  For $t \in (0,1)$, the derivative $\dot{u}(t)$
satisfies
\begin{align*}
\dot{u}(t) &= \frac{n\sqrt{t}}{4} \Expect \ntr\big[(\mtx{X} - \mtx{X}') \cdot \Delta^2 f(\mtx{Y}_t, \mtx{Y}_t', \mtx{Y}_t'')[ (\mtx{X} - \mtx{X}') \otimes (\mtx{X}' - \mtx{X}'') ] \big]\\
	&+ \frac{n\sqrt{t}}{4} \Expect \ntr\big[(\mtx{X} - \mtx{X}') \cdot \Delta^2 f(\mtx{Y}_t, \mtx{Y}_t'', \mtx{Y}_t'')[ (\mtx{X} - \mtx{X}'') \otimes (\mtx{X} - \mtx{X}') ] \big]
	\eqqcolon \onecirc + \twocirc.
\end{align*} 
The exchangeable counterparts $\mtx{Y}_t' \coloneqq \mtx{Y}_t(\mtx{X}', \mtx{Z})$
and $\mtx{Y}_t'' \coloneqq \mtx{Y}_t(\mtx{X}'', \mtx{Z})$.
\end{proposition}

\begin{proof}
\Cref{prop:mtx-interp} provides a formula~\eqref{eqn:mtx-mom-interp} for the derivative $\dot{u}$
along the interpolation path. We use the integration by parts rules to control the two
expectations.  The Gaussian IBP rule (\cref{prop:gauss-ibp-mtx}) with $\mtx{F}(\mtx{Z}) \gets f(\mtx{Y}_t(\mtx{X}, \mtx{Z}))$
yields
\[
\frac{1}{2\sqrt{1- t}} \Expect \ntr[ (\mtx{Z}-\Expect\mtx{Z}) f(\mtx{Y}_t(\mtx{X}, \mtx{Z})) ]
	= \frac{1}{2} \ip{ \Expect \Diff f( \mtx{Y}_t(\mtx{X}, \mtx{Z}) ) }{ \Varo[ \mtx{X} ] }.
\]
We have employed the chain rule and the hypothesis that $\Varo[ \mtx{X} ] = \Varo[ \mtx{Z} ]$.
Meanwhile, the discrete IBP rule (\cref{thm:discrete-ibp-mtx}) with inputs %
$\mtx{Y}(\mtx{X}) \gets \mtx{Y}_t(\mtx{X}, \mtx{Z})$ and $\alpha \gets \sqrt{t}$
provides that
\begin{align*}
\frac{1}{2\sqrt{t}} \Expect \ntr[ & (\mtx{X} - \Expect \mtx{X}) f(\mtx{Y}_t(\mtx{X}, \mtx{Z})) ] \\
	&= 	\frac{1}{2} \ip{ \Expect \Diff f( \mtx{Y}_t(\mtx{X}, \mtx{Z}) ) }{ \Varo[ \mtx{X} ] } \\
	&+ \frac{n\sqrt{t}}{4} \Expect \ntr\big[(\mtx{X} - \mtx{X}') \cdot \Delta^2 f(\mtx{Y}_t, \mtx{Y}_t', \mtx{Y}_t'')[ (\mtx{X} - \mtx{X}') \otimes (\mtx{X}' - \mtx{X}'') ] \big] \\
	&+ \frac{n\sqrt{t}}{4} \Expect \ntr\big[(\mtx{X} - \mtx{X}') \cdot \Delta^2 f(\mtx{Y}_t, \mtx{Y}_t'', \mtx{Y}_t'')[ (\mtx{X} - \mtx{X}'') \otimes (\mtx{X} - \mtx{X}') ] \big].
\end{align*}
Introduce the two displays into~\eqref{eqn:mtx-mom-interp} to see that the derivative
terms cancel.  
\end{proof}

\section{Universality for the Cauchy transform}
\label{sec:cauchy-xform}

We are now prepared to develop universality results for
specific trace functions.  This section treats universality
for the Cauchy transform, and the upcoming sections deal
with monomial moments (\cref{sec:mono-mom})
and the norms of the resolvents (\cref{sec:resolvent-norm}).

\subsection{Random matrix statistics}

Recall the model~\eqref{eqn:indep-sum-mtx} for the independent
sum $\mtx{X}$ of random self-adjoint matrices and the description~\eqref{eqn:gauss-proxy-pf}
of the Gaussian proxy $\mtx{Z}$.  The comparison of these models naturally leads
to statements that are expressed in terms of more detailed statistics.
Let us collect the definitions here for future reference.

First, we introduce the sum of third moments:
\begin{equation} \label{eqn:M3-mtx}
M_3(\mtx{X}) \coloneqq \sum_{i=1}^n \Expect \ntr \abs{\mtx{S}_i - \mtx{S}_i'}^3,
\end{equation}
where $(\mtx{S}_1', \dots, \mtx{S}_n')$ is an independent copy of
the sequence $(\mtx{S}_1, \dots, \mtx{S}_n)$ of summands.
For each natural number $p \in \N$,
define parameterized versions of the matrix variance and the uniform bound parameter:
\begin{equation} \label{eqn:mvar-p}
\sigma_{2p}^2(\mtx{X}) \coloneqq \lnorm{ \Expect \big[ (\mtx{S}_i - \Expect \mtx{S}_i)^2 \big] }_{p}
\quad\text{and}\quad
L_{2p}(\mtx{X}) \coloneqq \norm{ \max\nolimits_i \norm{\mtx{S}_i - \mtx{S}_i'} }_{2p}.
\end{equation}
We also allow the value $p = \infty$.
By monotonicity of $\set{L}_p$ norms and the triangle inequality, we have the bounds $\sigma_{2p}(\mtx{X}) \leq \sigma(\mtx{X})$ and $L_{2p}(\mtx{X}) \leq 2L(\mtx{X})$.
The statistics in this paragraph are typically easy to calculate using linear algebra arguments.

We will also prove more precise results,
stated in terms of another second moment statistic:
\begin{equation} \label{eqn:M2p-mtx}
M_{2,p}(\mtx{X}) \coloneqq \lnorm{ \sum_{i=1}^n (\mtx{S}_i - \mtx{S}_i')^2 }_p.
\end{equation}
The matrix Rosenthal inequality (\cref{fact:rosenthal-mtx}, positive case) allows
us to control $M_{2,p}$ in terms of more accessible statistics.
For each natural number $p \in \N$,
\begin{equation} \label{eqn:M2p-ros}
M_{2, p}(\mtx{X}) \leq 4 \sigma_{2p}^2(\mtx{X}) + 4p L_{2p}(\mtx{X}).
\end{equation}
The statement~\eqref{eqn:M2p-ros} follows from a short calculation, akin to \cref{cor:scalar-poly}.

\subsection{Universality for the Cauchy transform}

With this preparation behind us, we can state the main result of this section.
We will establish that the Cauchy transform of an independent sum %
of random matrices is consistent with the Cauchy transform of the Gaussian proxy. %

\begin{theorem}[Cauchy transform: Universality] \label{thm:univ-cauchy}
Let $\mtx{X}$ be an independent sum~\eqref{eqn:indep-sum-mtx} of random self-adjoint matrices,
and construct the matching Gaussian matrix $\mtx{Z}$ as in~\eqref{eqn:gauss-proxy-pf}.
For each complex number $\zeta \in \C \setminus \R$, the Cauchy transform $G_{\zeta}$
satisfies the comparison
\[
\abs{ G_\zeta( \mtx{X} ) - G_\zeta( \mtx{Z}) }
	\leq \frac{{M}_3(\mtx{X})}{\abs{\Im \zeta}^4}.
\]
The Cauchy transform $G_{\zeta}$ is defined in~\eqref{eqn:cauchy-xform},
and the statistic $M_3(\mtx{X})$ is defined in~\eqref{eqn:M3-mtx}.
\end{theorem}

\Cref{thm:univ-cauchy} coincides with a result alluded to
in~\cite[Rem.~6.13]{BvH24:Universality-Sharp}.
We can derive the statement of \cref{thm:cauchy-xform} by making
the simple bound $M_3(\mtx{X}) \leq 4 \sigma^2(\mtx{X}) L(\mtx{X})$.

\subsection{Proof of \cref{thm:univ-cauchy}}

Fix a complex number $\zeta \in \C \setminus \R$ with nonzero imaginary part,
and let $a \in \R$ be a real variable.
The Cauchy transform $G_{\zeta}$ is the trace function associated with
the resolvent function $R_\zeta(a) \coloneqq (\zeta - a)^{-1}$.
Note that the resolvent function is rational, and it has no
pole on the real line.  Its derivative is
$f(a) \coloneqq \Diff R_\zeta(a) = R_\zeta(a)^2$.

From \cref{ex:inv-power-second-diff}, we easily
obtain the second matrix difference of $f = R_\zeta^2$.  For self-adjoint matrices
$\mtx{A}_0, \mtx{A}_1, \mtx{A}_2; \mtx{H}_1, \mtx{H}_2 \in \Sym_d$,
\begin{align} 
\Delta^2 f(\mtx{A}_0, \mtx{A}_1, \mtx{A}_2)[ \mtx{H}_1 \otimes \mtx{H}_2 ]
	&= \sum_{q+r+s = 4}^+ (\zeta - \mtx{A}_0)^{-q} \mtx{H}_1 (\zeta - \mtx{A}_1)^{-r} \mtx{H}_2 (\zeta - \mtx{A}_2)^{-s} \notag \\
	&= \sum_{q+r+s = 4}^+ R_{\zeta}(\mtx{A}_0)^q \,\mtx{H}_1\, R_{\zeta}(\mtx{A}_1)^r\, \mtx{H}_2 \, R_{\zeta}(\mtx{A}_2)^s. \label{eqn:resolvent2-2diff}
\end{align}
The indices in the sum are \hilite{natural numbers}, so they range as
$(q,r,s) \in \{ (2,1,1); (1,2,1); (1,1,2) \}$.

We may assume that
the independent sum $\mtx{X}$ and the Gaussian proxy $\mtx{Z}$ are independent.
Define the interpolating matrix $\mtx{Y}_t$ as in~\eqref{eqn:mtx-interp},
and construct the interpolant for the Cauchy transform: %
\[
u(t) \coloneqq G_\zeta(\mtx{Y}_t) = \Expect \ntr R_\zeta(\mtx{Y}_t)
\quad\text{for $t \in [0,1]$.}
\]
Introduce the exchangeable counterparts $\mtx{X}', \mtx{X}''$ and $\mtx{Y}_t', \mtx{Y}_t''$ as in~\cref{thm:discrete-ibp-mtx}. %
\Cref{prop:mtx-interp-sum} furnishes a formula for the derivative $\dot{u}(t)$
along the interpolation path.

Consider the first of the two terms in $\dot{u}(t)$, as the second is entirely analogous.
The expression~\eqref{eqn:resolvent2-2diff} for the second matrix difference yields
\begin{align*}
\onecirc &\coloneqq \frac{n\sqrt{t}}{4} \Expect \ntr \big[ (\mtx{X} - \mtx{X}') \cdot \Delta^2 f(\mtx{Y}_t, \mtx{Y}_t', \mtx{Y}_t'')[ (\mtx{X} - \mtx{X}') \otimes (\mtx{X}' - \mtx{X}'') ] \big] \\
	&= \frac{n\sqrt{t}}{4} \sum_{q+r+s=4}^+ \Expect \ntr\big[ (\mtx{X} - \mtx{X}') \cdot R_{\zeta}(\mtx{Y}_t)^q \, (\mtx{X} - \mtx{X}') \, R_{\zeta}(\mtx{Y}_t')^r \, (\mtx{X}' - \mtx{X}'') \, R_{\zeta}(\mtx{Y}_t'')^{s} \big]. %
\end{align*}
Each of the resolvents is controlled uniformly:
\[
\norm{ R_{\zeta}(\mtx{A})^{p} }
\leq \abs{ \Im \zeta }^{-p}
\quad\text{for $\mtx{A} \in \Sym_d$ and $p \in \N$.}
\]
Bound each of the three summands in $\onecirc$ using H{\"o}lder's inequality for the expected trace with exponents
$(1/3, 0, 1/3, 0, 1/3, 0)$.  We reach
\[
\abs{\onecirc} \leq \frac{3\sqrt{t}}{4 \abs{\Im \zeta}^4} \cdot n \norm{ \mtx{X} - \mtx{X}' }_3 \norm{ \mtx{X} - \mtx{X}' }_3 \norm{ \mtx{X}' - \mtx{X}'' }_3
	= \frac{3 M_3(\mtx{X}) \sqrt{t}}{4 \abs{\Im \zeta}^4}.
\]
By exchangeability, each of the $\set{L}_3$ norms has the same value.  For example,
\[
\norm{ \mtx{X} - \mtx{X}' }_3
	\coloneqq \left( \Expect \ntr \abs{ \mtx{X} - \mtx{X}' }^3 \right)^{1/3}
	= \left( \frac{1}{n} \sum_{i=1}^n \Expect \ntr \abs{ \mtx{S}_i - \mtx{S}_i' }^3 \right)^{1/3}
	= \left(n^{-1} {M}_3(\mtx{X}) \right)^{1/3}.
\]
This is just a matter of definitions.

In summary, the derivative along the interpolation path satisfies the bound
\[
\abs{\dot{u}(t)} \leq \abs{ \onecirc } + \abs{ \twocirc } \leq \frac{3 {M}_3(\mtx{X}) \sqrt{t}}{2 \abs{\Im \zeta}^4}.
\]
It remains to integrate along the interpolation path.  By the fundamental theorem of calculus,
\[
\abs{ G_\zeta(\mtx{X}) - G_\zeta(\mtx{Z}) }
	= \abs{ u(1) - u(0) }
	\leq \int_0^1 \abs{\dot{u}(t)} \idiff{t}
	\leq \frac{3 {M}_3(\mtx{X})}{2 \abs{\Im \zeta}^4} \int_0^1 \sqrt{t}\idiff{t}
	= \frac{{M}_3(\mtx{X})}{\abs{\Im \zeta}^4}.
\]
This is the required result. 
\hfill\qed

\section{Universality for monomial moments}
\label{sec:mono-mom}

This section establishes a universality result for the even-order
monomial moments of an independent sum.  The statement and
argument parallel the scalar result (\cref{thm:scalar-poly}).

\begin{theorem}[Monomial moments: Universality] \label{thm:mom-univ}
Let $\mtx{X}$ be an independent sum~\eqref{eqn:indep-sum-mtx} of random self-adjoint matrices,
and construct the matching Gaussian matrix $\mtx{Z}$ as in~\eqref{eqn:gauss-proxy-pf}.
For each natural number $p \in \N$, the moments of order $2p$ satisfy the comparisons
\begin{equation} \label{eqn:mom-univ-cube}
\labs{ \norm{ \mtx{X} }_{2p} - \norm{ \mtx{Z} }_{2p} }^3
	\leq \big\vert \norm{ \mtx{X} }_{2p}^3 - \norm{ \mtx{Z} }_{2p}^3 \big\vert
	\leq (p-1) (2p-1) \cdot {M}_{2,p}(\mtx{X})  {L}_{2p}(\mtx{X}).
\end{equation}
In particular, the relative error admits the bound
\begin{equation} \label{eqn:mom-univ-rel}
\frac{\big\vert \norm{ \mtx{X} }_{2p} - \norm{ \mtx{Z} }_{2p} \big\vert}{\norm{\mtx{Z}}_{2p}}
	\leq \frac{(p-1)(2p-1) \cdot {M}_{2,p}(\mtx{X}) {L}_{2p}(\mtx{X})}{\norm{\mtx{Z}}_{2p}^3}.
\end{equation}
The statistics $L_{2p}$ and $M_{2,p}$ are defined
in~\eqref{eqn:mvar-p} and~\eqref{eqn:M2p-mtx}.
\end{theorem}

From here, we can obtain a result that depends on the matrix variance
statistic by invoking the matrix Rosenthal inequality (\cref{fact:rosenthal-mtx}).

\begin{corollary}[Monomial moments: Universality] \label{cor:mom-univ}
Instate the assumptions of \cref{thm:mom-univ}.  Then
\[
\labs{ \norm{ \mtx{X} }_{2p} - \norm{ \mtx{Z} }_{2p} }
	\leq \big( 8p^2 \sigma_{2p}^2(\mtx{X}) L_{2p}(\mtx{X}) \big)^{1/3} + 8p L_{2p}(\mtx{X}).
\]
Furthermore, when $p L_{2p}^2(\mtx{X}) \leq \sigma_{2p}^2(\mtx{X})$,
\[
\big\vert \norm{ \mtx{X} }_{2p} - \norm{ \mtx{Z} }_{2p} \big\vert
	\leq 16 p^2 L_{2p}(\mtx{X}). %
\]
The statistics $\sigma_{2p}$ and $L_{2p}$ are defined in \eqref{eqn:mvar-p}.
\end{corollary}

\begin{proof}[Proof sketch]
The first claim follows from~\eqref{eqn:mom-univ-cube} and the Rosenthal
bound for $M_{2,p}$ given in~\eqref{eqn:M2p-ros}.  The second claim
follows from the relative error bound~\eqref{eqn:mom-univ-rel},
the Rosenthal bound~\eqref{eqn:M2p-ros}, and the fact that
\[
\norm{ \mtx{Z} }_{2p}^2 \geq \norm{ \mtx{Z} - \Expect \mtx{Z} }_{2p}^2
	\geq \sigma_{2p}^2(\mtx{Z}).
\]
The first inequality holds because $\mtx{Z} - \Expect \mtx{Z}$ has a symmetric distribution,
while the second estimate follows from Jensen's inequality.
\end{proof}

\Cref{cor:mom-univ} is a variant of
the result~\cite[Thm.~2.9]{BvH24:Universality-Sharp} with a slightly
different choice of uniform bound statistic.
\Cref{thm:mom-intro} follows from~\cref{cor:mom-univ}, along with
simple bounds for the statistics.

\subsection{Third moment bound}

We begin with a lemma that indicates how the third moment
statistic $M_{2,p}$ arises in this context.

\begin{lemma}[Third moment bound] \label{lem:triple-product-bd}
Let $(\mtx{X}, \mtx{X}', \mtx{X}'')$ be exchangeable counterparts
as in \cref{sec:mtx-cov-id}.  For each natural number $p \geq 2$,
\[
n \cdot  %
\lnorm{ \Expect \big[ \norm{\mtx{X}' - \mtx{X}''} \, (\mtx{X} - \mtx{X}')^2 \lcondbar \mtx{X} \big] }_{2p/3}
	\leq {M}_{2,p}(\mtx{X}) {L}_{2p}(\mtx{X}).
\]
\end{lemma}

\begin{proof}
We begin by averaging over the random index $I$,
and then we invoke Jensen's inequality to remove
the rest of the conditioning:
\begin{multline*}
n \cdot \lnorm{ \Expect \big[ \norm{\mtx{X}' - \mtx{X}''} \, (\mtx{X} - \mtx{X}')^2 \lcondbar \mtx{X} \big] }_{2p/3} \\
	\leq \lnorm{ \sum_{i = 1}^n \norm{\mtx{S}_i' - \mtx{S}_i''} \, (\mtx{S}_i - \mtx{S}_i')^2 }_{2p/3}
	\leq \lnorm{ \max\nolimits_i \norm{ \mtx{S}_i' - \mtx{S}_i'' } \cdot \sum_{i=1}^n (\mtx{S}_i - \mtx{S}_i')^2 }_{2p/3} \\
	\leq \left( \Expect \max\nolimits_i \norm{ \mtx{S}_i - \mtx{S}_i' }^{2p} \right)^{1/(2p)}
	\lnorm{ \sum_{i=1}^n (\mtx{S}_i - \mtx{S}_i')^2 }_{p}
	= {L}_{2p}(\mtx{X}) {M}_{2,p}(\mtx{X}).
\end{multline*}
We controlled the sum by drawing out the maximum of the scalar coefficients,
using the matrix monotonicity of the $\set{L}_{2p/3}$ norm.
The last step follows from H{\"o}lder's inequality for the expectation
with exponents $\theta = 1/3$ and $\theta' = 2/3$, and we have used exchangeability to
simplify. %
\end{proof}
     
\subsection{Proof of \cref{thm:mom-univ}}

Since the independent sum $\mtx{X}$ and the Gaussian proxy $\mtx{Z}$
have matching first- and second-order moments,
the results are trivially correct when $p = 1$.

Fix a natural number $p \geq 2$.  The power function $h(a) \coloneqq a^{2p}$
has derivative $f(a) \coloneqq \Diff h(a) = 2p a^{2p-1}$.  From \cref{ex:power-second-diff},
we read off the second matrix difference of the function $f$.
For self-adjoint matrix arguments, we have
\begin{equation} \label{eqn:mtx-second-diff-power-deriv}
\Delta^2 f(\mtx{A}, \mtx{B}, \mtx{C})[ \mtx{G} \otimes \mtx{H} ]
	= 2p \sum_{q+r+s = 2p - 3} \mtx{A}^q \mtx{G} \mtx{B}^r \mtx{H} \mtx{C}^s.
\end{equation}
The indices $q,r,s$ in the sum range over \hilite{nonnegative integers}, so the total
number of summands %
is ${2p - 1 \choose 2} = (p-1)(2p-1)$.

Without loss of generality, assume that the random matrices $\mtx{X}$ and $\mtx{Z}$
are independent.  Define the interpolating matrix $\mtx{Y}_t$ as in~\eqref{eqn:mtx-interp},
and construct the interpolant for the power function:
\[
u(t) \coloneqq \Expect \ntr h(\mtx{Y}_t) = \Expect \ntr \mtx{Y}_t^{2p} = \norm{ \mtx{Y}_t }_{2p}^{2p}
\quad\text{for $t \in [0,1]$.}
\]
Introduce the exchangeable counterparts $\mtx{X}', \mtx{X}''$ and $\mtx{Y}_t', \mtx{Y}_t''$ as in~\cref{thm:discrete-ibp-mtx}. %
\cref{prop:mtx-interp-sum} controls the derivative $\dot{u}(t)$ on the interpolation path.

Consider the first term in $\dot{u}(t)$.
The expression~\eqref{eqn:mtx-second-diff-power-deriv}
for the second matrix difference gives
\begin{align*}
\onecirc &\coloneqq \frac{n\sqrt{t}}{4} \Expect \ntr\big[ (\mtx{X} - \mtx{X}') \cdot \Delta^2 f(\mtx{Y}_t, \mtx{Y}_t', \mtx{Y}_t'')[(\mtx{X} - \mtx{X}') \otimes (\mtx{X}' - \mtx{X}'') ] \big] \\ 	&=\frac{2pn\sqrt{t}}{4} \sum_{q+r+s = 2p-3} \Expect \ntr\big[ (\mtx{X} - \mtx{X}') \cdot (\mtx{Y}_t)^q (\mtx{X} - \mtx{X}') (\mtx{Y}_t')^r (\mtx{X}' - \mtx{X}'') (\mtx{Y}_t'')^s \big].
\end{align*}
Define the random matrices $\mtx{W}_0 = \mtx{W}_1 = \mtx{X} - \mtx{X}'$ and $\mtx{W}_2 = \mtx{X}' - \mtx{X}''$.
For each of the ${2p-1 \choose 2}$ choices for the indices $(q, r, s)$, %
apply the matrix consolidation inequality for exchangeable random matrices (\cref{prop:mtx-solid-exch})
to reach
\[
\abs{ \onecirc } \leq \frac{2p n\sqrt{t}}{4} {2p - 1 \choose 2} \cdot \max\nolimits_{\pi, j, k}
	\Expect \ntr\big[ \Expect \big[ \norm{\mtx{W}_j^\pi} \, (\mtx{W}_k^\pi)^2 \lcondbar \mtx{X} \big] \cdot \abs{ \mtx{Y}_t }^{2p-3} \big].
\]
To decouple the matrices $\mtx{W}_j$ from the interpolant $\mtx{Y}_t$,
invoke H{\"o}lder's inequality with exponents $\theta = 3/(2p)$ and $\theta' = (2p-3)/(2p)$.  Thus,
\begin{align*}
\abs{ \onecirc } &\leq \frac{2p \sqrt{t}}{4} {2p - 1 \choose 2} \cdot \max\nolimits_{\pi, j, k}
	n \cdot \lnorm{ \Expect \big[ \norm{\mtx{W}_j^\pi} \, (\mtx{W}_k^\pi)^2 \lcondbar \mtx{X} \big] }_{2p/3}
	\cdot \norm{ \mtx{Y}_t }_{2p}^{1 - 3/(2p)} \\
	&\leq \frac{2p (2p-1)(2p-2) \sqrt{t}}{8} \cdot {M}_{2,p}(\mtx{X}) {L}_{2p}(\mtx{X}) \cdot u(t)^{1 - 3/(2p)}.
\end{align*}
The precise calculation depends on the choice of $(\pi, j, k)$ in the maximum,
but the steps are identical with \cref{lem:triple-product-bd}.
This expression exposes a copy of the interpolant $u(t)$,
so it amounts to a moment comparison.
The second term $\twocirc$ in $\dot{u}(t)$ admits the same bound.

Combining the bounds for $\onecirc$ and $\twocirc$, we arrive at the differential inequality
\[
\abs{ \dot{u}(t) }
	\leq \abs{\onecirc} + \abs{\twocirc}
	\leq \frac{2p (2p-1)(2p-2) \sqrt{t}}{4} \cdot {M}_{2,p}(\mtx{X}) {L}_{2p}(\mtx{X}) \cdot u(t)^{1 - 3/(2p)}.
\]
To solve the inequality, we integrate the function $u(t)^{3/(2p)}$
to reach %
\begin{align*}
\abs{ u(1)^{3/(2p)} - u(0)^{3/(2p)} }
	\leq \frac{3}{2p} \int_0^1 \abs{ u(t) }^{3/(2p) - 1} \cdot \abs{ \dot{u}(t) } \idiff{t} %
	\leq (p-1)(2p-1) \cdot {M}_{2,p}(\mtx{X}) {L}_{2p}(\mtx{X}).
\end{align*}
Finally, note that $u(1)^{1/(2p)} = \norm{ \mtx{X} }_{2p}$
and $u(0)^{1/(2p)} = \norm{ \mtx{Z} }_{2p}$.
The remaining bounds follow from the numerical inequalities~\eqref{eqn:cube-ineqs}.
\hfill\qed

\section{Universality for resolvent norms}
\label{sec:resolvent-norm}

Finally, we turn to the universality result for the norm
of the resolvent matrix.

\begin{theorem}[Resolvent norm: Universality] \label{thm:resolvent-norm-univ}
Let $\mtx{X}$ be an independent sum~\eqref{eqn:indep-sum-mtx} of random self-adjoint matrices,
and construct the matching Gaussian matrix $\mtx{Z}$ as in ~\eqref{eqn:gauss-proxy-pf}.
For each complex number $\zeta \in \C \setminus \R$ and each natural number $p \in \N$,
the norm of the resolvent $R_{\zeta}$ satisfies the comparison
\[
\labs{ \norm{ R_{\zeta}(\mtx{X}) }_{2p} - \norm{ R_{\zeta}(\mtx{Z} ) }_{2p} }
	\leq \frac{ (p+1)(2p+1) \cdot {M}_{2,p}(\mtx{X}) {L}_{\infty}(\mtx{X})}{\abs{\Im \zeta}^4}.
\]
The statistics $L_{\infty}$ and $M_{2,p}$ are defined
in~\eqref{eqn:mvar-p} and~\eqref{eqn:M2p-mtx}.
\end{theorem}

Applying the matrix Rosenthal inequality (\cref{fact:rosenthal-mtx})
as in~\eqref{eqn:M2p-ros}, we can obtain estimates
that depend on the matrix variance statistic.

\begin{corollary}[Resolvent norm: Universality] \label{cor:resolvent-norm-univ}
Instate the assumptions of \cref{thm:resolvent-norm-univ}.
Then
\[
\labs{ \norm{ R_{\zeta}(\mtx{X}) }_{2p} - \norm{ R_{\zeta}( \mtx{Z} ) }_{2p} }
	\leq \frac{ (p+1)(2p+1) \cdot \big( 4 \sigma_{2p}^2(\mtx{X}) L_{\infty}(\mtx{X}) + 4 p L_{\infty}^3(\mtx{X})\big)}{\abs{\Im \zeta}^4}.
\]
\end{corollary}

\Cref{cor:resolvent-norm-univ} is a variant of~\cite[Thm.~6.8]{BvH24:Universality-Sharp}
with a slightly different choice of upper bound statistic.
\Cref{thm:resolvent-norm} follows from~\cref{cor:resolvent-norm-univ}
after some simple bounds on the statistics and a coarse bound on the
factors involving $p$.

\subsection{Resolvent powers: Second matrix differences}

To prepare, we need to compute the second matrix difference
of the resolvent power.
Fix a complex number $\zeta \in \C \setminus \R$ with nonzero imaginary part.
Introduce the resolvent function $R_{\zeta}(a) \coloneqq (\zeta - a)^{-1}$ for $a \in \R$.
For a self-adjoint matrix $\mtx{A}$, abbreviate its resolvent:
$\mtx{A}_{\zeta} \coloneqq R_\zeta(\mtx{A}) = (\zeta - \mtx{A})^{-1}$.
Note that $\mtx{A}_{\zeta}$ is a normal matrix
with adjoint $(\mtx{A}_{\zeta})^* = \mtx{A}_{\zeta^*}$.

For a natural number $p \in \N$, we are interested in the resolvent power
\[
h(a) \coloneqq \abs{R_{\zeta}(a)}^{2p}
	= R_{\zeta^*}(a)^p R_{\zeta}(a)^p.
\]
Since $\Diff R_{\zeta}(a) = R_{\zeta}(a)^{2}$, we have
\[
f(a) \coloneqq (\Diff h)(a) = p \cdot R_{\zeta^*}(a)^{p+1} R_{\zeta}(a)^p
	+ p \cdot  R_{\zeta}(a)^{p+1} R_{\zeta^*}(a)^p.
\]
This function is more complicated than the ones we have considered before,
but we can quickly obtain its second differences by breaking down the calculation.

We will argue that
\begin{equation} \label{eqn:resolvent-second-diff}
\Delta^2 f(\mtx{A}, \mtx{B}, \mtx{C}) %
	= p \cdot \sum\nolimits_{q_i, r_i, s_i \in \set{Q}} \mtx{A}_{\zeta^*}^{q_1} \mtx{A}_{\zeta}^{q_2} \otimes %
	\mtx{B}_{\zeta^*}^{r_1} \mtx{B}_{\zeta}^{r_2} \otimes %
	\mtx{C}_{\zeta^*}^{s_1} \mtx{C}_{\zeta}^{s_2}.
\end{equation}
In this expression, $\set{Q}$ is a multiset consisting of sextuples
$(q_1, q_2, r_1, r_2, s_2, s_3)$ of nonnegative integers.
The cardinality $\# \set{Q} = 2(p+1)(2p + 1)$, and each sextuple sums to $2p + 3$.
For legibility, we use the tensor product $\otimes$ as a placeholder for the direction
arguments in the matrix second difference.

The (matrix) second difference of $f$ takes the form
\begin{equation} \label{eqn:res-pow-2diff}
\Delta^2 f = p\cdot  \Delta^2 \big(R_{\zeta^*}^{p+1} R_{\zeta}^p \big) + p \cdot \Delta^2 \big(R_{\zeta}^{p+1} R_{\zeta^*}^p \big).
\end{equation}
To calculate each of the two matrix second differences
appearing in~\eqref{eqn:res-pow-2diff},
we begin with \cref{ex:inv-power-second-diff}.
This result implies that
\begin{align*}
R_{\zeta}\left(
\begin{bmatrix}
\mtx{A} & \mtx{G} & \\
& \mtx{B} & \mtx{H} \\
&& \mtx{C}
\end{bmatrix} \right)^{p} %
= \begin{bmatrix}
	\mtx{A}_{\zeta}^p &
	\sum_{q+r = p + 1}^+ \mtx{A}_{\zeta}^q \mtx{G} \mtx{B}_{\zeta}^r &
	\sum_{q+r+s = p + 2}^+ \mtx{A}_{\zeta}^q \mtx{G} \mtx{B}_{\zeta}^r \mtx{H} \mtx{C}_{\zeta}^s \\
	& \mtx{B}_{\zeta}^p &
	\sum_{r+s = p + 1}^+ \mtx{B}_{\zeta}^r \mtx{H} \mtx{C}_{\zeta}^s \\
	&& \mtx{C}_{\zeta}^p
\end{bmatrix}.
\end{align*}
In this expression, the indices $q,r,s$ range over \hilite{natural numbers}.
An analogous formula holds for the matrix second difference of
the function $R_{\zeta^*}(\cdot)^{p+1}$.
Multiplying out these two block matrices and extracting the northeast block,
\begin{align*} \label{eqn:resolvent-second-diff-part}
\Delta^2 \big( R_{\zeta^*}^{p+1} R_{\zeta}^p \big)(\mtx{A}, \mtx{B}, \mtx{C}) %
	&= \sum_{q+r+s = p+2}^+ \mtx{A}_{\zeta^*}^{p+1}\mtx{A}_{\zeta}^q \otimes  \mtx{B}_{\zeta}^r \otimes  \mtx{C}_{\zeta}^s \\
	&+ \sum_{q+r=p+2}^+ \sum_{r'+s=p+1} \mtx{A}_{\zeta^*}^q \otimes  \mtx{B}_{\zeta^*}^r \mtx{B}_{\zeta}^{r'} \otimes  \mtx{C}_{\zeta}^s \\
	&+ \sum_{q+r+s = p+3}^+ \mtx{A}_{\zeta^*}^{q} \otimes  \mtx{B}_{\zeta^*}^r \otimes  \mtx{C}_{\zeta^*}^s \mtx{C}_{\zeta}^{p}.
\end{align*}
In each product, the total of the powers is $2p + 3$, but the number of adjoints varies.
The total number of summands across the three lines is $(p+1)(2p + 1)$.
The other term in the second difference~\eqref{eqn:res-pow-2diff}
takes the same form with $\zeta \rightarrow \zeta^*$.
These considerations lead to the result~\eqref{eqn:resolvent-second-diff}.

\subsection{Proof of \cref{thm:resolvent-norm-univ}}

We may assume that the independent sum $\mtx{X}$ and the Gaussian proxy $\mtx{Z}$
are independent.
Define the interpolating matrix $\mtx{Y} \coloneqq \mtx{Y}(t)$,
as in~\eqref{eqn:mtx-interp}.
The interpolant of the resolvent power is
\[
u(t) \coloneqq \Expect \ntr h(\mtx{Y}(t)) = \Expect \ntr \abs{R_{\zeta}(\mtx{Y}(t))}^{2p}
	= \norm{ {R}_{\zeta}(\mtx{Y}(t)) }_{2p}^{2p}.
\]
We will suppress the time variable $t$ for the rest of the calculation.
Introduce the exchangeable counterparts $\mtx{X}', \mtx{X}''$ and $\mtx{Y}', \mtx{Y}''$ as in~\cref{thm:discrete-ibp-mtx}. %
\Cref{prop:mtx-interp} furnishes the derivative $\dot{u}(t)$ along the interpolation path.
The bound for $\dot{u}(t)$ combines the techniques from the universality
for resolvents and for moments. 

Consider the first component $\onecirc$ of the derivative $\dot{u}(t)$.
In view of~\eqref{eqn:resolvent-second-diff},
\begin{align*}
\onecirc &\coloneqq \frac{n \sqrt{t}}{4} \Expect \ntr \big[ (\mtx{X} - \mtx{X}') \cdot
	\Delta^2 f(\mtx{Y}, \mtx{Y}', \mtx{Y}'')[ (\mtx{X} - \mtx{X}') \otimes (\mtx{X}' - \mtx{X}'') ] \big] \\
	&= \frac{pn\sqrt{t}}{4} \sum_{\set{Q}} \Expect \ntr \big[
	(\mtx{X} - \mtx{X}') \cdot (\mtx{Y}_{\zeta^*})^{q_1} (\mtx{Y}_{\zeta})^{q_2} \,
	 (\mtx{X} - \mtx{X}') \, (\mtx{Y}_{\zeta^*}')^{r_1} (\mtx{Y}_{\zeta}')^{r_2} \, %
	(\mtx{X}' - \mtx{X}'') \, (\mtx{Y}_{\zeta^*}'')^{s_1} (\mtx{Y}_{\zeta}'')^{s_2} \big].
\end{align*}
This expression looks formidable, but it is not any more challenging than before.
Each pair $\mtx{Y}_{\zeta}$ and $\mtx{Y}_{\zeta^*}$ consists of commuting
normal matrices, and $\abs{ \mtx{Y}_{\zeta} } = \abs{ \mtx{Y}_{\zeta^*} }$.
Thus, we can repeat the proof of \cref{prop:mtx-solid-exch} without any
essential change to consolidate the powers.
Since $\#\set{Q} = 2(p+1)(2p+1)$ and the powers sum to $2p + 3$,
we arrive at the bound
\[
\abs{\onecirc} \leq \frac{2p(p+1)(2p+1) \sqrt{t}}{4} %
	\max\nolimits_{\pi, j,k} n \cdot \Expect \ntr \big[ \Expect \big[ \norm{ \mtx{W}_j^\pi} \, (\mtx{W}_k^\pi)^2 \lcondbar \mtx{X} \big]
	\cdot \abs{ \mtx{Y}_{\zeta} }^{2p+3} \big].
\]
As before, $\mtx{W}_0 = \mtx{W}_1 = \mtx{X} - \mtx{X}'$ and $\mtx{W}_2 = \mtx{X}' - \mtx{X}''$.

To continue, we must decouple the random matrices in the expected trace.
For this purpose, we factor out $\abs{\mtx{Y}_{\zeta}}^4$,
and we apply H{\"o}lder's inequality with the exponents $(1/2p, (2p-1)/(2p), 0)$.  Thus,
\begin{align*}
\abs{\onecirc} &\leq \frac{2p(p+1)(2p+1) \sqrt{t}}{4} \cdot \max\nolimits_{\pi,j,k}
	n \cdot \lnorm{ \Expect \big[ \norm{ \mtx{W}_j^\pi} \, (\mtx{W}_k^\pi)^2 \lcondbar \mtx{X} \big] }_{2p}
	 \cdot \norm{ \mtx{Y}_{\zeta} }_{2p}^{2p - 1}\cdot \norm{ \mtx{Y}_{\zeta} }_{\infty}^4  \\
	&\leq \frac{2p(p+1)(2p+1) \sqrt{t}}{4} \cdot \frac{ {M}_{2,p}(\mtx{X}) {L}_{\infty}(\mtx{X})}{\abs{\Im \zeta}^4} \cdot u(t)^{1-1/(2p)}.
\end{align*}
The bound for the maximum follows the same pattern as before
(cf.~\cref{lem:triple-product-bd}),
as does the trivial spectral norm bound for the resolvent.  We have located
a copy of the interpolant $u(t)$.  The second term
$\twocirc$ in $\dot{u}(t)$ satisfies the identical bound.

In summary, we have shown that
\[
\abs{ \dot{u}(t) }
	\leq \abs{\onecirc} + \abs{\twocirc}
	\leq \frac{2p(p+1)(2p+1) \sqrt{t}}{2} \cdot \frac{ {M}_{2,p}(\mtx{X}) {L}_{\infty}(\mtx{X})}{\abs{\Im \zeta}^4} \cdot u(t)^{1-1/(2p)}.
\]
From here, the end of the proof is the same as always.
\begin{align*}
\labs{ u(1)^{1/(2p)} - u(0)^{1/(2p)} }
	\leq \frac{1}{2p} \int_{0}^1 \abs{u(t)}^{1/(2p) - 1} \abs{\dot{u}(t)} \idiff{t}
	\leq (p+1)(2p+1) \cdot \frac{ {M}_{2,p}(\mtx{X}) {L}_{\infty}(\mtx{X})}{\abs{\Im \zeta}^4}.
\end{align*}
Finally, recall the meanings of $u(1)$ and $u(0)$.  %
\hfill\qed

\appendix

\section{Scalar divided differences: Further properties}
\label{app:scalar-diff}

Divided differences are obviously related to simple differences of function values.  For example,
\begin{equation} \label{eqn:div-diff-diff}
f(a) - f(b) = (a-b) \cdot (\Delta f)(a,b)
\quad\text{for all $a, b \in \R$.}
\end{equation}
By iteration, we can expand the difference to higher order.
In its general form, this result is called the Taylor--Taylor theorem~\cite[Thm.~4.1]{KV14:Foundations-Free}.

\begin{fact}[Scalar Taylor--Taylor] \label{fact:taylor-taylor}
Consider a smooth function $f : \R \to \R$, and fix two points $a, b \in \R$.
Then
\[
f(a) - f(b) = \sum_{p=1}^k \frac{(a-b)^p}{p!} \cdot (\Diff^p f)(b)
	+ (a-b)^{k+1} \cdot \Delta^{k+1}(a, b, \dots, b).
\]
In the divided difference $\Delta^{k+1}(a, b, \dots, b)$, the second argument $b$ is repeated $k$ times.
\end{fact}

The first divided difference is clearly a symmetric function of its arguments:
\[
(\Delta f)(a, b) = (\Delta f)(b, a)
\quad\text{for all $a, b \in \R$.}
\]
By iteration, higher-order divided differences $(\Delta^k f)(a_0, \dots, a_k)$
are also symmetric in the arguments $(a_0, \dots, a_k)$.

Divided differences support calculus rules that resemble the
rules for manipulating derivatives.  These results follow
from the definition and the continuity properties of the
divided differences.

\begin{fact}[Divided differences: Calculus] \label{fact:diff-calculus}
Consider smooth functions $f, g : \R \to \R$, and fix points $a, b \in \R$.
For the composition, the chain rule states that
\[
(\Delta (f \circ g))(a, b) = (\Delta f)(g(a), g(b)) \cdot (\Delta g)(a, b).
\]
For the product, the Leibniz rule states that
\[
(\Delta (f \cdot g))(a,b) = f(a) \cdot (\Delta g)(a,b) + (\Delta f)(a,b) \cdot g(b).
\]
Higher-order divided differences satisfies similar rules.
\end{fact}

Divided differences admit an integral representation, called the
Genocchi--Hermite formula, that exposes additional properties.
For $k \in \N$, the \term{probability simplex} of dimension $k$ is the set
\[
\set{T}_k \coloneqq \{ \vct{\tau} \geq \vct{0} : \ip{ \vct{1} }{\vct{\tau}} = 1 \} \subset \R^{k+1}.
\]
The simplex $\set{T}_k$ has volume $(k!)^{-1}$.

\begin{fact}[Divided differences: Integral representation] \label{fact:diff-integral}
Let $f : \R \to \R$ be a smooth function.  For $k \in \N$
and points $a_0, \dots, a_{k} \in \R$, the divided difference
can be expressed as 
\[
(\Delta^{k} f)(a_0, \dots, a_{k}) = \int_{\set{T}_{k}} (\Diff^k f)(\tau_0 a_0 + \dots + \tau_k a_k) \idiff{\vct{\tau}}.
\]
In particular, $(\Delta^k f)(a, \dots, a) = (\Diff^k f)(a) / k!$.  Furthermore,
\begin{equation} \label{eqn:diff-deriv-bd}
\norm{ \Delta^k f }_{\sup} \leq \frac{1}{k!} \cdot \norm{ \Diff^k f }_{\sup}.
\end{equation}
\end{fact}

If the function $\abs{ (\Diff^k f ) }$ is \hilite{convex}, then we can bound the divided
difference by the derivative at one of the arguments:
\[
\abs{ (\Delta^k f)(a_0, \dots, a_k) } \leq \frac{1}{k!} \max\nolimits_{i=0,\dots,k} \abs{ (\Diff^k f)(a_i) }.
\]
Indeed, the integrand achieves its maximum at an extreme point of the simplex,
namely one of the standard basis vectors.

\section{Rosenthal inequalities}
\label{app:rosenthal}

Rosenthal inequalities~\cite{Ros70:Subspaces-Lp}
give bounds for the $\set{L}_p$ norm of an independent
sum in terms of the properties of the summands.
This appendix present variants in the spirit of
results developed by
Nagaev, Pinelis, and Utev~\cite{NP77:Some-Inequalities,PU84:Estimates-Moments}
in the scalar case.
The more recent matrix extensions are adapted
from Mackey et al.~\cite{MJCFT14:Matrix-Concentration}.

\subsection{Scalar Rosenthal inequalities}

In the scalar setting, the Rosenthal inequalities are well established.
Here is a version with explicit constants.

\begin{fact}[Rosenthal: Scalar case] \label{fact:rosenthal-scalar}
Consider an independent sum $X  \coloneqq \sum_{i=1}^n S_i$ of scalar random variables.

\begin{itemize} \setlength{\itemsep}{4pt}
\item	\textbf{Positive case.}  Assume that the summands $S_i$ are \hilite{positive}.
For $p = 1$ and $p \geq 1.5$,
\[
\norm{ X }_{2p} \leq \left[ (\Expect X)^{1/2} + \sqrt{2p-1} \cdot \norm{ \max\nolimits_i S_i }_{2p}^{1/2} \right]^{2}.
\]
When $1 < p < 1.5$, the constant factor on the maximum is at most $2$.

\item	\textbf{Centered case.}  Assume that the summands $S_i$ are \hilite{centered}.
For $p = 1/2$ and $p = 1$ and $p \geq 1.5$,
\[
\norm{ X }_{4p} \leq \sqrt{4p-1} \cdot (\Expect X^2)^{1/2} + \sqrt{(4p-1)(2p-1)} \cdot \norm{ \max\nolimits_i \abs{S_i} }_{4p}.
\]
\end{itemize}
\end{fact}

The positive case is the monomial moment analog of the Chernoff inequality,
while the centered case is the analog of the Bernstein inequality.
For completeness, we offer a simple proof based on a
BDG-type inequality~\cite[Thm.~1.5(iii)]{Cha07:Steins-Method}.

\begin{fact}[BDG inequality] \label{fact:bdg-scalar}
Let $Y \coloneqq \sum_{i=1}^n W_i$ be an independent sum of centered, real random variables.
Define the variance proxy $V_Y \coloneqq \frac{1}{2} \sum_{i=1}^n (W_i^2 + \Expect W_i^2)$.
For $p = 1$ and $p \geq 1.5$,
\[
\norm{ Y }_{2p} \leq \sqrt{2p-1} \cdot \lnorm{ V_Y }_{p}^{1/2}.
\]
When $1 < p < 1.5$, the constant is at most $\sqrt{4p-2} \leq 2$.
\end{fact}

\begin{proof}[Proof of \cref{fact:rosenthal-scalar}]
We begin with the case $X  \coloneqq \sum_{i=1}^n S_i$ where the summands 
are positive: $S_i \geq 0$.
Introduce the centered independent sum $Y  \coloneqq X - \Expect X$.
Its variance proxy satisfies
\[
V_Y = \frac{1}{2} \sum_{i=1}^n \big( (S_i - \Expect S_i)^2 + \Expect (S_i - \Expect S_i)^2 \big)
	\leq \frac{1}{2} \sum_{i=1}^n (S_i^2 + \Expect S_i^2).
\]
The latter inequality relies on positivity.
For $p = 1$ and $p \geq 1.5$, the BDG inequality (\cref{fact:bdg-scalar}) yields
\begin{align*}
\norm{ Y }_{2p}
	&\leq \sqrt{2p-1} \cdot \lnorm{ \frac{1}{2} \sum_{i=1}^n (S_i^2 + \Expect S_i^2) }_{p}^{1/2} \\
	&\leq \sqrt{2p-1} \cdot \lnorm{ \sum_{i=1}^n S_i^2 }_{p}^{1/2}
	\leq \sqrt{2p-1} \cdot \lnorm{ \max\nolimits_i \abs{S_i} \cdot \sum_{i=1}^n S_i }_{p}^{1/2} \\
	&\leq \sqrt{2p-1} \cdot \norm{ \max\nolimits_i \abs{S_i} }_{2p}^{1/2} \cdot \lnorm{ \sum_{i=1}^n S_i }_{2p}^{1/2}
	= \sqrt{2p-1} \cdot \norm{ \max\nolimits_i \abs{S_i} }_{2p}^{1/2} \cdot \norm{ X }_{2p}^{1/2}.
\end{align*}
The second line follow from the triangle inequality and Jensen's inequality;
the third line depends on Cauchy--Schwarz.
Returning to the original sum $X$, we center and invoke the last display:
\begin{align*}
E^2 \coloneqq \norm{ X }_{2p} \leq \abs{\Expect X} + \norm{ X - \Expect X}_{2p}
	\leq  \abs{\Expect X} + \sqrt{2p-1} \cdot \norm{ \max\nolimits_i \abs{S_i} }_{2p}^{1/2} \cdot E.
\end{align*}
Solve the quadratic inequality to reach the stated result.

Now, consider the case $Y  \coloneqq \sum_{i=1}^n W_i$ where the summands $W_i$ are centered.
Its variance proxy satisfies
\[
V_Y = \frac{1}{2} \sum_{i=1}^n (W_i^2 + \Expect W_i^2).
\]
For $p = 1/2$ and $p = 1$ and $p \geq 1.5$, we can apply the BDG inequality (\cref{fact:bdg-scalar})
to reach
\[
\norm{ Y }_{4p} \leq \sqrt{4p-1} \cdot \lnorm{ \frac{1}{2} \sum_{i=1}^n (W_i^2 + \Expect W_i^2) }_{2p}^{1/2}
	\leq \sqrt{4p-1} \cdot \lnorm{ \sum_{i=1}^n W_i^2 }_{2p}^{1/2}.
\]
This is an independent sum of positive random variables, so we can apply the first result:
\[
\norm{ Y }_{4p} \leq \sqrt{4p-1} \left[ \abs{ \Expect Y^2 }^{1/2} + \sqrt{2p-1} \cdot \norm{ \max\nolimits_i W_i^2 }_{2p}^{1/2} \right].
\]
Rewrite this inequality to obtain the stated result.
\end{proof}

\subsection{Matrix Rosenthal inequalities}

Scalar Rosenthal inequalities extend to matrices with the same proof concept,
although the constants may be slightly different.

\begin{fact}[Rosenthal: Matrix case] \label{fact:rosenthal-mtx}
Consider an independent sum $\mtx{X}  \coloneqq \sum_{i=1}^n \mtx{S}_i$
of random self-adjoint matrices.  Let $\norm{\cdot}_p$ denote
the normalized $\set{L}_p$ norm, as in~\eqref{eqn:Lq-norm}.  %

\begin{itemize} \setlength{\itemsep}{4pt}
\item	\textbf{Positive case.}  Assume that the summands $\mtx{S}_i$ are \hilite{positive semidefinite}.
For $p = 1$ and $p \geq 1.5$,
\[
\norm{\mtx{X}}_{2p} \leq \left[ \norm{ \Expect \mtx{X} }_{2p}^{1/2}
	+ \sqrt{4p-2} \cdot \norm{ \max\nolimits_i \norm{\mtx{S}_i}}_{2p}^{1/2} \right]^2
\]
When $1 < p < 1.5$, the constant attached to the maximum is at most
$\sqrt{8p-4} < 4$.

\item	\textbf{Centered case.}  Assume that the summands $\mtx{S}_i$ are \hilite{centered}.
For $p = 1/2$ and $p = 1$ and $p \geq 1.5$,
\[
\norm{\mtx{X}}_{4p} \leq \sqrt{4p-1} \cdot \norm{ (\Expect \mtx{X}^2)^{1/2} }_{4p}
	+ \sqrt{(4p-1)(4p-2)} \cdot \norm{ \max\nolimits_i \norm{\mtx{S}_i}}_{4p}.
\]
\end{itemize}
\end{fact}

Much like the matrix Rosenthal inequality presented in~\cite[Cor.~7.4]{MJCFT14:Matrix-Concentration},
this statement follows from the matrix BDG inequality~\cite[Thm.~7.1]{MJCFT14:Matrix-Concentration}.
The variant stated here depends on an alternative bound for positive-semidefinite summands:
\[
\lnorm{ \sum_i \mtx{S}_i^2 }_{p} \leq \lnorm{ \max\nolimits_i \norm{\mtx{S}_i} \cdot \sum_i \mtx{S}_i }_{p}
	\leq \norm{ \max\nolimits_i \norm{\mtx{S}_i} }_{2p}^{1/2} \cdot \lnorm{ \sum_i \mtx{S}_i }_{2p}^{1/2}.
\]
This estimate exploits the fact that the normalized Schatten norms are increasing in the order $p$.
The second inequality is just Cauchy--Schwarz.  The rest of the details are the same
as in the paper~\cite{MJCFT14:Matrix-Concentration}, which mirrors the scalar
argument presented in the last subsection.

\section*{Acknowledgments}

JAT would like to thank Eitan Levin for introducing him to the
matrix difference calculus described in \cref{sec:mtx-calculus}.
Ramon van Handel provided valuable feedback on an early draft of this
manuscript.
The research was supported in part by NSF FRG Award \#1952777,
ONR BRC Award N-00014-18-1-2363, ONR Award N-00014-24-1-2223,
and the Caltech Carver Mead New Adventures fund.

\printbibliography[title=References]

\end{document}